\documentclass{article}
    \usepackage{color}
    \usepackage[centertags]{amsmath}
    \usepackage{amsfonts} 
    \usepackage{fullpage} 
    \usepackage{mathrsfs}
    \usepackage{amsthm}
    \usepackage{amssymb}
    \usepackage{graphicx}
    \usepackage{hyperref}
    \usepackage{epsfig}

\newtheorem{thm}{Theorem}[section]

\newtheorem{lem}[thm]{Lemma}
\newtheorem{prop}[thm]{Proposition}
\theoremstyle{definition}
\newtheorem{defn}[thm]{Definition}
\theoremstyle{remark}
\newtheorem{rem}[thm]{Remark}
\newcommand{\R}{\mathbb R} 

\newcommand{\E}{\mathbb E}

\newcommand{\tr}{\text{tr}}
\newcommand{\Tr}{\text{Tr}}

\renewcommand\Im{\operatorname{\mathfrak{Im}}}

\numberwithin{equation}{section}
\begin{document}
\title{Central Limit Theorem for Linear Statistics of Eigenvalues of Band 
Random Matrices}

\author{Lingyun Li
\thanks{Department of Mathematics, University of California, Davis, 
One Shields Avenue, Davis, CA 95616-8633, llyli@math.ucdavis.edu}
\and
Alexander Soshnikov
\thanks{Department of Mathematics, University of California, Davis, 
One Shields Avenue, Davis, CA 95616-8633, soshniko@math.ucdavis.edu;
research has been supported in part by the NSF grant DMS-1007558}}
\maketitle

\begin{abstract}
We prove the Central Limit Theorem for linear statistics of the eigenvalues of band random matrices provided $\sqrt{n} \ll b_n \ll n$ 
and test functions are sufficiently smooth.
\end{abstract}
\section{Introduction}
The goal of this paper is to prove the Central Limit Theorem for linear statistics of the eigenvalues of real symmetric band random matrices
with independent entries.  

First, we define a real symmetric band random matrix.
Let $\{b_n\}$ be a sequence of integers satisfying $0\leq b_n\leq n/2$ such that $b_n\to \infty$ as $n \to \infty.$
Define
\begin{align}
& d_n(j,k):=\min\{|k-j|,n-|k-j|\}, \\
& I_n:=\{(j,k): d_n(j,k)\leq b_n,\ j,k=1,...,n\}, \ \ \text{and} \ \ I_n^+:=\{(j,k): (j,k)\in I_n, \ j\leq k\}.
\end{align}
In particular, $d_n$ has the following natural interpretation: if the first $n$ positive integers are evenly spread out on a circle of radius 
$\frac{n}{2\*\pi}$,  
then $d_n(j,k)$ is the distance between the integers $j$ and $k$.

The quantity $b_n$ will be the $radius$ of a band of our random matrix.  In other words, 
all entries of the matrix with ${j,k} \notin I_n$ are going to be zero.
We define a real symmetric band random matrix 
\begin{equation}
M=(M_{jk}), \ 1\leq j,k\leq n,
\end{equation}
in such a way that for $j\leq k$ one has
\begin{equation}
\label{Mdef1}
M_{jk}=M_{kj}=b_n^{-1/2}\*W_{jk} \ \ \text{if} \ \ d_n(j,k)\leq b_n,
\end{equation}
and $M_{jk}=0$ otherwise, 
where $\{W_{jk}\}_{(j,k)\in I_n^+}$ is a sequence of independent real valued random variables satisfying
\begin{equation}
\mathbb{E}\{W_{jk}\}=0,\ \mathbb{E}\{W_{jk}^2\}=(1+\delta_{jk})
\sigma^2.
\label{2m}
\end{equation}
In general, the distribution of the entries $W_{jk}$ might depend on the size $n$ of the matrix but we will not indicate this dependence 
in our notations, unless it is necessary.
An important special case corresponds to $b_n=\lfloor(n-1)/2\rfloor.$  Then $M$ is standard Wigner random 
matrix (see e.g. \cite{Wig},  \cite{B}, \cite{BG}, \cite{AGZ}). 

For a real symmetric (Hermitian) matrix $M$ of order $n,$ its  
empirical distribution of the eigenvalues is defined as 
$\mu_M = \frac{1}{n} \sum_{i=1}^{n} \delta_{\lambda_i},$ where  $\lambda_1 \leq \ldots \leq \lambda_n$ are the (ordered) eigenvalues of $M.$
The Wigner semicircle law states that for any bounded continuous test function $\varphi: \R \to \R,$ the linear statistic 
\begin{equation}
 \frac{1}{n} \sum_{i=1}^n \varphi(\lambda_i) = \frac{1}{n} \*\Tr(\varphi(M))=:\tr_n(\varphi(M)) 
\end{equation}
converges to $\int \varphi(x) \* d \mu_{sc}(dx) $ in probability, where $\mu_{sc}$ is determined by its density

\begin{equation}
\label{polukrug}
\frac{d \mu_{sc}}{dx}(x) = \frac{1}{4 \pi \sigma^2} \sqrt{ 8 \sigma^2 - x^2} \mathbf{1}_{[-2\*\sqrt{2}\* \sigma , 2\*\sqrt{2}\* \sigma]}(x).
\end{equation}

We refer the reader to \cite{Wig},  \cite{B}, \cite{BG}, \cite{AGZ} for the proof in the full matrix case and to \cite{semicirclelawband2},
\cite{semicirclelawband1} for the proof in the band matrix case.

Band random matrices  have important applications in physics (see e.g. \cite{Selig}, \cite{Cas1}, \cite{Cas2}, \cite{Mirlin1}, \cite{Mirlin2}, 
\cite{Spen}), in particular as a model of quantum chaos.  It is conjectured that the eigenvectors are localized and local eigenvalue statistics
are Poisson for $b_n\ll\sqrt{n}.$  On the other hand, it is expected that the eigenvectors are delocalized and local eigenvalue statistics follow GUE 
(GOE) law for $b_n\gg\sqrt{n}$ ( see e.g. \cite{Mirlin1}).  Throughout the paper, the relation $a_n\ll b_n$ for two $n$-dependent quantities $a_n$ and 
$b_n$ means that $a_n/b_n\to0$ as $n\to\infty$. For recent mathematical progress on local spectral properties of band random matrices, 
we refer the reader to \cite{Yau1}, \cite{Yau2}, \cite{Yau3}, \cite{Schen}, \cite{Sodin}, \cite{Sodin1}.

The linear eigenvalues statistics corresponding to a test function $\varphi$ is defined as
\begin{equation}
\mathscr{N}_n[\varphi]=\sum_{l=1}^n\varphi(\lambda_l).
\end{equation}
In the Wigner (full matrix) case, the variance of $\mathscr{N}_n[\varphi]$ stays bounded as $n\to \infty$ for sufficiently smooth $\varphi.$  Moreover,
the fluctuation of the linear statistic is Gaussian in the limit (see e.g.  \cite{sashasinai1}, 
\cite{Bandmodel}, \cite{B.spectral}, \cite{P.CLT}, \cite{Shcherbina1}, and references therein). Similar results have been established for other 
ensembles of random matrices (\cite{Johansson1}, \cite{sasha4}, \cite{Shcherbina2},\cite{BS}). In addition, we note recent results on partial linear 
eigenvalue statistics (\cite{BPZ}, \cite{orourke}) and the properties of the eigenvectors of Wigner matrices (\cite{BaiPan}).

In this paper, we prove that the normalized linear statistic
\begin{equation}
\label{wings}
\mathscr{M}_n[\varphi]:=(b_n/n)^{1/2}\mathscr{N}_n[\varphi]
\end{equation}
has an asymptotic normal distribution, as $n\to\infty$ provided $b_n \gg \sqrt{n},$ and $\varphi,\ W_{jk}$ satisfy some conditions. 

\section{Statement of Main Results}

For the first theorem, we assume that the matrix entries satisfy the Poincar\'e inequality.
We refer the reader to Section A of the Appendix for the definition and basic facts about the Poincar\'e inequality.
\begin{thm}
\label{thm-Poin}
Let $M=W/\sqrt{b_n}$ be a real symmetric random band matrix (\ref{Mdef1}-\ref{2m}), where 
$\{b_n\}$ is a sequence of integers satisfying $\sqrt{n}\ll b_n\ll n$. 
Assume the following:
\begin{enumerate}
\item Diagonal and non-zero off-diagonal entries of $W$ are two sets of i.i.d random variables;
\item  The marginal probability distribution 
of $W_{jk}$ satisfies the Poincar\'e Inequality with some uniform constant $m>0$ which does not depend on $n, j, k$;

\item 
The fourth moment of the non-zero off-diagonal entries does not depend on $n$:
\begin{equation}
\mu_4=\mathbb{E}\{W_{12}^4\}.
\end{equation}

\end{enumerate}
Let $\varphi:\mathbb{R}\to\mathbb{R}$ be a test function with continuous bounded derivative.
Then the corresponding centered normalized linear statistic of the eigenvalues
\begin{equation}
\label{linejnaya}
\mathscr{M}_n^{\circ}[\varphi]:=(b_n/n)^{1/2}\mathscr{N}_n^{\circ}[\varphi]=(b_n/n)^{1/2}(\mathscr{N}_n[\varphi]-\mathbb{E}\{\mathscr{N}_n[\varphi]\})
\end{equation}
converges in distribution to the Gaussian random variable with zero mean and the variance
\begin{eqnarray}
\label{varvar}
Var_{band}[\varphi]&=&\int_{-2\sqrt{2}\sigma}^{2\sqrt{2}\sigma}
\int_{-2\sqrt{2}\sigma}^{2\sqrt{2}\sigma}\int_{-2\sqrt{2}\sigma}^{2\sqrt{2}\sigma}
\*\frac{(\varphi(x)-\varphi(\lambda))\varphi'(y)\*\sqrt{8\sigma^2-x^2}\*\sqrt{8\sigma^2-y^2}}{4\pi^4(x-\lambda)
\*\sqrt{8\sigma^2-\lambda^2}}\*F_{\sigma}(x,y)\* 1_{\{x\not=y\}}\*dx\*dy\*d\lambda\nonumber\\
&&+\frac{\kappa_4}{16\pi^2\sigma^8}\left(\int_{-2\sqrt{2}\sigma}^{2\sqrt{2}\sigma}
\frac{\varphi(\lambda)(4\sigma^2-\lambda^2)}{\sqrt{8\sigma^2-\lambda^2}}d\lambda\right)^2,\label{varb} 
\end{eqnarray}
where for $x\not=y$
\begin{equation}
F_{\sigma}(x,y):=\int_{-\infty}^{\infty}
\frac{\left(s^3\*\sin s-s\*\sin^3 s\right)ds}{2\sigma^2\left(s^2-\sin^2 s\right)^2-\left(s^3\sin s+s\sin^3 s\right)xy+s^2\sin^2 s(x^2+y^2)},
\end{equation}
and $\kappa_4$ is the fourth cumulant of off-diagonal entries, i.e.
\begin{equation}
\kappa_4=\mu_4-3\sigma^4.
\end{equation}
\end{thm}
Next, we extend this result to the non-i.i.d. case when the fifth moment of the matrix entries is uniformly bounded.  Here we do not assume that 
marginal 
distributions of the non-zero entries satisfy the Poincar\'e inequality. For technical reasons, we assume that the 
fourth cumulant of the matrix entries is zero.  Also we require that $\sqrt{n}\ln n\ll b_n$ 
(thus, we have additional $\ln n$ factor at the l.h.s. as compared to the corresponding assumption in Theorem \ref{thm-Poin}).

\begin{thm}
\label{thm:0cumulant}
Let $M=W/\sqrt{b_n}$ be a real symmetric band matrix (\ref{Mdef1}-\ref{2m}), where
$\{b_n\}$ is a sequence of positive integers satisfying $\sqrt{n}\ln n\ll b_n\ll n$. 
Assume the following:
\begin{enumerate}
\item 
\begin{equation}\sigma_5:=\sup_{n\in\mathbb{N}}\max_{(j,k)\in I_n}\mathbb{E}\{|W_{jk}^{(n)}|^5\}<\infty.\end{equation}
\item The third cumulant of the non-zero off-diagonal entries does not depend on $j,k$: 
$$\kappa_3=\kappa_{3,jk}, \ \ (j,k)\in I_n, \ j\neq k.$$
\item The fourth cumulant is zero: $\kappa_4=\mathbb{E}\{(W_{jk}^{(n)})^4\}-3\sigma^4=0$.
\end{enumerate}
Let $\varphi:\mathbb{R}\to\mathbb{R}$ be a test function with the Fourier transform
\begin{equation}
\hat{\varphi}(t)=\frac{1}{2\pi}\int_{-\infty}^{\infty} e^{-it\lambda}\varphi(\lambda)d\lambda
\label{def:fouriertransform}
\end{equation}
satisfying
\begin{equation}
\int_{-\infty}^{\infty} (1+|t|^{4})|\hat{\varphi}(t)|dt<\infty.
\label{ineq:fourier1}
\end{equation}
Then the corresponding centered normalized linear eigenvalues statistic $\mathscr{M}_n^{\circ}[\varphi]$ converges in distribution to the Gaussian 
random variable with zero mean and variance $Var_{G}$ 
\begin{equation}
\label{Gausvar}
Var_{G}[\varphi]=\int_{-2\sqrt{2}\sigma}^{2\sqrt{2}\sigma}
\int_{-2\sqrt{2}\sigma}^{2\sqrt{2}\sigma}\int_{-2\sqrt{2}\sigma}^{2\sqrt{2}\sigma}
\*\frac{(\varphi(x)-\varphi(\lambda))\varphi'(y)\*\sqrt{8\sigma^2-x^2}\sqrt{8\sigma^2-y^2}}{4\pi^4(x-\lambda)
\*\sqrt{8\sigma^2-\lambda^2}}F_{\sigma}(x,y)\* 1_{\{x\not=y\}} \*dx\*dy \*d\lambda.
\end{equation}
\end{thm}

\begin{rem}
{\it It should be noted that in \cite{Khorunzhy} the authors claimed to compute the asymptotic formula for the variance of the trace of the resolvent
of a band random matrix (see the formulas (3.5)-(3.7) therein).  In particular, they claim that the (normalized) variance has the same 
limiting expression as 
in the GOE case. We disagree with this statement.  In fact, it is not hard to see that the limit of $\frac{b_n}{n} Var Tr M^3$ is different 
in the band and the full Wigner matrix cases.

In our computations, the difference is highlighted by the fact that the limiting behavior of the expression (\ref{An}) is different in the band 
and the full matrix cases.  In the full Wigner case, the formula for the 
limit of $A_n(t)$ immediately follows from the Wigner semicircle law.  In the band case, 
the fact 
that the summation in (\ref{An}) is restricted to $(j,k): d_n(j,k)\leq b_n$ leads to a different limit formula (see Subsection \ref{an}, in particular 
Proposition \ref{propan} and Lemma \ref{lemma35}.}
\end{rem}

\begin{rem}
{\it Similar results with little modification hold for Hermitian band random matrices.  In particular, the variance (\ref{Gausvar}) in Theorem 
\ref{thm:0cumulant} gets an additional 
factor $1/2$, provided  (\ref{2m}) is replaced by
\begin{equation}
\mathbb{E}\{W_{jk}\}=0,\ \mathbb{E}\{|W_{jk}|^2\}=(1+\delta_{jk})\*\sigma^2,
\ \ \mathbb{E}\{W_{jk}^2\}=0.
\label{2mm}
\end{equation}

In addition, it should be noted that the results of Theorems \ref{thm-Poin} and \ref{thm:0cumulant} hold if one replaces the condition 
$M_{jk}=0$ for $d_n(j,k)>b_n$ by  $M_{jk}=0$ for $|j-k|>b_n.$

The proofs are very similar and left to the reader.}
\end{rem}

The rest of the paper is organized as follows.  We prove Theorem \ref{thm-Poin} in Section 3 and Theorem \ref{thm:0cumulant} in Section 4.  In the 
Appendix, we list basic facts about the Poincar\'e inequality and decoupling formula.


\section{Proof of Theorem 2.1}

\subsection{Stein's Method}
We follow the approach used by A. Lytova and L. Pastur in \cite{P.CLT} in the full matrix (Wigner) case. 
Essentially, it is a modification of the Stein's method (\cite{Stein}, \cite{BarChen}).
While several steps of our proof are similar to the ones
 in \cite{P.CLT}, 
the fact that we are dealing with band matrices raises new significant difficulties 
(see e.g. Lemmas \ref{lemma34} and \ref{lemma35} in Subsection \ref{an}).

First, we prove the result of Theorem \ref{thm-Poin} under an additional technical condition on the smoothness of a test function.
Namely, we assume that the Fourier transform of $\varphi:\mathbb{R}\to\mathbb{R}$
satisfies
\begin{equation}
\int_{-\infty}^{\infty} (1+|t|^{4+\varepsilon})|\hat{\varphi}(t)|dt<\infty,
\label{ineq:fourier}
\end{equation}
where $\varepsilon$ is an arbitrary small positive number. Once the result is established for such test functions, it can be easily
extended to the case of functions with bounded continuous derivative using (\ref{pi1}).

Let $Z_n(x),Z(x)$ be the characteristic functions of the normalized linear statistic 
(\ref{wings}) and the Gaussian distribution with zero mean and 
$Var_{band}[\varphi]$ 
variance, respectively, i.e.
\begin{equation}
Z_n(x)=\mathbb{E}\{e^{ix\mathscr{M}_n^{\circ}[\varphi]}\}
\label{Zn},
\end{equation}
and
\begin{equation}
Z(x)=\exp\{-x^2Var_{band}[\varphi]/2\}.
\end{equation}
It is sufficient to show that for any $x\in\mathbb{R}$
\begin{equation}
\lim_{n\to\infty}Z_n(x)=Z(x).
\end{equation}
We note that $Z(x)$ is the unique solution of the integral equation
\begin{equation}
Z(x)=1-Var_{band}[\varphi]\int_0^xyZ(y)dy
\label{eqn:Zn1}
\end{equation}
in the class of bounded continuous functions. 
It follows from (\ref{Zn}) that the derivative of $Z_n(x)$ can be written as
\begin{equation}
Z_n'(x)=i\mathbb{E}\{\mathscr{M}_n^{\circ}[\varphi]e^{ix\mathscr{M}_n^{\circ}[\varphi]}\}.
\end{equation}
To bound the derivative of $Z_n,$ we use the Poincar\'e inequality. Since the Poincar\'e Inequality tensorises (see e.g. \cite{AGZ}),
the joint distribution of $\{W_{jk}\}_{(j,k)\in I_n^+}$ on $\mathbb{R}^{n(b_n+1)}$ satisfies the 
Poincar\'e Inequality with the same constant $m>0$ , i.e. for all continuously differentiable function $\Phi$, we have
\begin{equation}
Var\{\Phi(\{W_{jk}\}_{(j,k)\in I_n^+})\}\leq \frac{1}{m}\*
\sum_{(j,k)\in I_n^+}\mathbb{E}\{|\frac{\partial\Phi}{\partial W_{jk}}(\{W_{jk}\})|^2\}\label{ppi}.
\end{equation}
Let
\begin{equation}
\beta_{jk}=(1+\delta_{jk})^{-1}=\left\{\begin{array}{cc}1&j\neq k,\\1/2&j=k.\end{array}\right.
\label{beta}
\end{equation}
Since
\begin{equation}
\frac{\partial\mathscr{M}_n[\varphi]}{\partial W_{jk}}=\frac{2\beta_{jk}}{\sqrt{n}}\varphi'_{jk}(M),
\label{lem1eqn}
\end{equation}
we have
\begin{eqnarray}
Var\{\mathscr{M}_n[\varphi]\}&\leq&\frac{2}{mn}\mathbb{E}\{Tr(\varphi'(M)\varphi'(M)^*)\} \nonumber\\
&\leq&\frac{2}{m}(\sup_{\lambda\in\mathbb{R}}|\varphi'(\lambda)|)^2.
\label{pi1}
\end{eqnarray}
Applying the Cauchy-Schwarz inequality, we obtain
\begin{equation}
|Z_n'(x)|\leq\sqrt{\frac{2}{m}}\sup_{\lambda\in\mathbb{R}}|\varphi'(\lambda)|.
\end{equation}
In addition, (\ref{pi1}) implies
\begin{equation}
|Z_n''(x)|\leq \frac{2}{m}(\sup_{\lambda\in\mathbb{R}}|\varphi'(\lambda)|)^2.
\end{equation}
Taking into account $Z_n(0)=1$, we have
\begin{equation}
Z_n(x)=1+\int_0^xZ_n'(y)dy.
\end{equation}
We note that the sequence $\{(Z_n(x), Z_n'(x))\}$ is pre-compact in $C([-T, T], \R^2)$ for any $T>0.$
Therefore, it is enough to show that for any converging subsequence one has
\begin{equation}
\lim_{n_j\to\infty}Z'_{n_j}(x)=-xVar_{band}[\varphi]\lim_{n_j\to\infty}Z_{n_j}(x).
\label{eqn:limitZn}
\end{equation}
For the convenience of the reader, we use the same notations as in \cite{P.CLT}:
\begin{equation}
D_{jk}:=\partial/\partial M_{jk};\end{equation}
\begin{equation}
U(t):=e^{itM},U_{jk}(t):=(U(t))_{jk};
\label{ujk}
\end{equation}
\begin{equation}
u_n(t):=TrU(t), \ u_n^{\circ}(t):=u_n(t)-E\{u_n(t)\}.
\label{moskva}
\end{equation}

Since $U(t)$ is a unitary matrix, we have
\begin{equation}
\|U\|=1;\ |U_{jk}|\leq 1; \ \sum_{k=1}^n|U_{jk}|^2=1.
\label{Uprop}
\end{equation}
Moreover,
\begin{equation}
D_{jk}U_{ab}(t)=i\beta_{jk}(U_{aj}*U_{bk}+U_{ak}*U_{bj})(t),
\label{dif}
\end{equation}
where 
\begin{equation}
\label{svertka}
f*g(t):=\int_0^t f(s)\*g(t-s)\*ds.
\end{equation}
Applying the Fourier inversion formula
\begin{equation}
\varphi(\lambda)=\int_{-\infty}^{\infty} e^{it\lambda}\hat{\varphi}(t)dt,
\end{equation}
we can write
\begin{equation}
\mathscr{M}_n^{\circ}[\varphi]=(b_n/n)^{1/2}\int_{-\infty}^{\infty}\hat{\varphi}(t)u_n^{\circ}(t)dt.
\end{equation}
Therefore,
\begin{equation}
Z'_n(x)=i\int_{-\infty}^{\infty}\hat{\varphi}(t)Y_n(x,t)dt,
\label{eqn:Znprime}
\end{equation}
where 
\begin{equation}
Y_n(x,t):=\mathbb{E}\{(b_n/n)^{1/2}u_n^{\circ}(t)e_n(x)\},
\label{yn}
\end{equation}
and
\begin{equation}e_n(x)=e^{ix\mathscr{M}_n^{\circ}[\varphi]}.
\end{equation}
Taking into account (\ref{eqn:limitZn}) and (\ref{eqn:Znprime}), 
we conclude that the result of the theorem follows if we can establish the following two facts.
First, we have to show that the sequence $\{Y_n\}$ 
is bounded and equicontinuous on any bounded subset of $\{t\geq0,x\in\mathbb{R}\}.$
Second, we have to show that
any uniformly converging subsequence of $Y_n$ has the same 
limit 
\begin{equation}Y(x,t)=\overline{Y(-x,-t)}\end{equation}
such that
\begin{equation}
i\int_{-\infty}^{\infty}\hat{\varphi}(t)Y(x,t)dt=-xVar_{band}[\varphi]Z(x).
\end{equation}

The main technical part of the proof of Theorem \ref{thm-Poin} is the following proposition.
\begin{prop}
\label{ynprop}
$Y_n(x,t)$ satisfies the equation 
\begin{eqnarray}
\lefteqn{Y_n(x,t)+\frac{2(2b_n+1)\sigma^2}{b_n}\int_0^t\int_0^{t_1}\bar{v}_n(t_1-t_2)Y_n(x,t_2)dt_2dt_1}\nonumber\\
&&=xZ_n(x)A_n(t)+2i\kappa_4xZ_n(x)\int_0^t\bar{v}_n*\bar{v}_n(t_1)dt_1\int_{-\infty}^{\infty}t_2\bar{v}_n*\bar{v}_n(t_2)\hat{\varphi}(t_2)dt_2+
r_n(x,t),
\label{equationyn}
\end{eqnarray}
where
\begin{eqnarray}
\label{An}
& & A_n(t):=-\frac{2\sigma^2}{n}\int_0^t\sum_{(j,k)\in I_n}\mathbb{E}\{U_{jk}(t_1)\varphi_{jk}'(M)\}dt_1, \\
\label{vnt}
& & \bar{v}_n(t):=n^{-1}\* \mathbb{E} Tr e^{itM},
\end{eqnarray}
$U_{jk}(t)$ is defined in (\ref{ujk}), 
and $r_n(x,t)$ converges to zero uniformly on any bounded subset of $\{t\geq0,x\in\mathbb{R}\}$.
\end{prop}
The proof of Proposition \ref{ynprop} will be given in the remaining part of this subsection and in the next three subsections.

\begin{proof}

First, we show that $Y_n(x,t)$ is bounded and uniformly equicontinuous on bounded subsets of $\R^2.$
Indeed, applying inequality (\ref{pi1}) to $\varphi(\lambda)=e^{it\lambda}$ and $\varphi(\lambda)=i\lambda e^{it\lambda}$, we get
\begin{equation}
Var\{(b_n/n)^{1/2}u_n(t)\}\leq\frac{2t^2}{m}
\end{equation}
and
\begin{eqnarray}
Var\{(b_n/n)^{1/2}u'_n(t)\}\leq\frac{2}{m}(1+3\sigma^2t^2).
\end{eqnarray}
This implies
\begin{equation}|Y_n(x,t)|\leq Var^{1/2}\{(b_n/n)^{1/2}u_n(t)\}\leq\sqrt{\frac{2}{m}}|t|,
\end{equation}
\begin{equation}
|\frac{\partial}{\partial t}Y_n(x,t)|\leq Var^{1/2}\{(b_n/n)^{1/2}u_n'(t)\}\leq\sqrt{\frac{2}{m}(1+3\sigma^2t^2)},
\end{equation}
and
\begin{eqnarray}
|\frac{\partial}{\partial x}Y_n(x,t)|\leq\frac{2}{m}|t|\sup_{\lambda\in \mathbb{R}}|\varphi'(\lambda)|.
\end{eqnarray}
Therefore, we have shown that $\{Y_n\}$ is bounded and equicontinuous on any bounded subset of $\mathbb{R}^2$.
Applying the identity $e^{itM}=1+i\int_0^tMe^{isM}ds$, we have
\begin{equation}u_n(t)=n+i\int_0^t\sum_{(j,k)\in I_n}M_{jk}U_{jk}(t_1)dt_1,
\end{equation}
and
\begin{equation}
\label{spartak}
Y_n(x,t)=\frac{i}{\sqrt{n}}\int_0^t\sum_{(j,k)\in I_n}\mathbb{E}\{W_{jk}U_{jk}(t_1)e_n^{\circ}(x)\}dt_1,
\end{equation}
where $e_n^{\circ}=e_n-\mathbb{E}\{e_n\}$. 
To analyze (\ref{spartak}), we use the 
decoupling formula (\ref{decf}) with $p=3$ to obtain
\begin{equation}
\label{yyy}
Y_n(x,t)=\frac{i}{\sqrt{n}}\int_0^t\sum_{(j,k)\in I_n}\left\{\sum_{l=0}^3\frac{\kappa_{l+1,jk}}{b_n^{l/2}l!}
\mathbb{E}\{D_{jk}^l(U_{jk}(t_1)e_n^{\circ}(x))\}+\varepsilon_{3,jk}\right\}dt_1,
\end{equation}
where $\kappa_{l,jk}$ is the $l$th cumulant of $W_{jk}$, i.e. 
\[\kappa_{1,jk}=0,\ \kappa_{2,jk}=(1+\delta_{jk})\sigma^2;\]
in addition, for $j\neq k$ one has
\[\kappa_{3,jk}=\mathbb{E}\{(W_{12}^{(n)})^3\}=:\kappa_3,\ \kappa_{4,jk}=\kappa_4;\]
and for $j=k$
\[\kappa_{3,jj}=\mathbb{E}\{(W_{11}^{(n)})^3\}=:\kappa_3',\ \kappa_{4,jj}=\mathbb{E}\{(W_{11}^{(n)})^4\}-12\sigma^2=:\kappa_4'.\]
Moreover, we note that the remainder term $\varepsilon_{3,jk}$ in (\ref{yyy}) is bounded as
\begin{equation}
|\varepsilon_{3,jk}|\leq C_3\mathbb{E}\{|W_{jk}|^5\}\sup_{W_{jk}\in\mathbb{R}}\left|\frac{D_{jk}^4U_{jk}(t)e_n^{\circ}(x)}{b_n^2}\right|.
\label{eps}
\end{equation}
Since the marginal distribution of the matrix entries satisfies the PI with constant $m$ independent of $n$, 
the third and fourth cumulants are uniformly bounded in $n$, i.e. there exist $\sigma_3$ and $\sigma_4$ independent of $n$ such that
\begin{equation}
|\kappa_3|,|\kappa_3'|\leq \sigma_3,|\kappa_4'|\leq \sigma_4,
\label{inequal:cumulant}
\end{equation}
and $\sigma_5:=\max_{j,k,n} \mathbb{E} \{|W_{jk}|^5\}<\infty.$

We need the following technical lemma.
\begin{lem}
\label{difb}
\begin{equation}
|D_{jk}^l(U_{jk}(t)e_n^{\circ}(x))|\leq C_l(\sqrt{b_n/n}\*x,t),\ l=1,2,3,4,
\end{equation}
where $C_l(x,t)$ is some polynomial in $|x|,|t|$ of degree $l$ with positive coefficients independent of $n.$
\end{lem}
\begin{proof}
(\ref{dif}) implies 
\begin{equation}
|D_{jk}^lU_{jk}(t)|\leq c_l|t|^l.
\label{boundofdujk}
\end{equation}
In addition,
\begin{equation}
D_{jk}e_n(x)=-2(b_n/n)^{1/2}\beta_{jk}xe_n(x)\int_{-\infty}^{\infty} sU_{jk}(s)\hat{\varphi}(s)ds=2i(b_n/n)^{1/2}\beta_{jk}xe_n(x)\varphi_{jk}'(M).
\end{equation}
It follows from (\ref{ineq:fourier}) that $\varphi$ has fourth bounded derivative.
Thus, for $l=1,2,3,4$
\begin{equation}
|D_{jk}^le_n(x)|\leq c'_l(1+|(b_n/n)^{1/2}x|^l)\label{boundofden}.
\end{equation}
Combining (\ref{boundofdujk}) and (\ref{boundofden}) we obtain Lemma \ref{difb}.\qedhere
\end{proof}
Lemma \ref{difb} and (\ref{eps}) imply
\begin{equation}
|\varepsilon_{3jk}|\leq C_3\sigma_5\*C_4((b_n/n)^{1/2}x,t)/b_n^2.
\label{epsilonb}
\end{equation}
We can rewrite (\ref{yyy}) as
 \begin{equation}
\label{urav}
Y_n(x,t)=T_1+T_2+T_3+\mathscr{E}_3,
\end{equation}
where
\begin{equation}
T_l:=\frac{i}{l!\sqrt{nb_n^l}}\int_0^t\sum_{(j,k)\in I_n}\kappa_{l+1,jk}\mathbb{E}\{D_{jk}^l(U_{jk}(t_1)e_n^{\circ}(x))\}dt_1
\label{t1}, \ l=1,2,3,
\end{equation}
and
\begin{equation}
\mathscr{E}_3=in^{-1/2}\int_0^t\sum_{(j,k)\in I_n}\varepsilon_{3,jk}dt_1.
\end{equation}
By (\ref{epsilonb}), we have
\begin{equation}
\label{uravn}
|\mathscr{E}_3|\leq\frac{\sqrt{n}}{b_n}C_5((b_n/n)^{1/2}x,t).
\end{equation}
Since $n/b_n^2\to0$, we obtain that $\mathscr{E}_3\to0$ on any bounded subset of $\mathbb{R}^2$ as $n\to\infty.$
In the next three subsections, we consider separately each of the terms $T_l,\ l=1,2,3$ in (\ref{urav}) and finish the proof of Proposition \ref{ynprop}.

\subsection{Estimate of $T_1$}
\label{TT1}
The main result of this subsection is contained in the following proposition.

\begin{prop}
\label{T1}
Let $T_1$ be defined as in (\ref{t1}) with $l=1.$
Then
\begin{equation}
T_1= -\frac{2(2b_n+1)\sigma^2}{b_n}\int_0^t\int_0^{t_1}\bar{v}_n(t_1-t_2)Y_n(x,t_2)dt_2dt_1 + xZ_n(x)A_n(t)+\varepsilon_n(x,t),
\end{equation}
where 
$\bar{v}_n(t)$ is defined in (\ref{vnt}), 
$A_n(t)$
is defined in (\ref{An}),
and $\varepsilon_n(x,t) \to 0$ as $n\to \infty$ uniformly on any bounded subset of $\{(x,t),t\geq0\}.$
\end{prop}

\begin{proof}

First, by (\ref{dif}) we write
\begin{equation}
T_1=T_{11}+T_{12}+T_{13},
\end{equation}
where
\begin{eqnarray}
T_{11}=-\frac{\sigma^2}{\sqrt{nb_n}}\int_0^t\sum_{(j,k)\in I_n}\mathbb{E}\{U_{jk}*U_{jk}(t_1)e_n^{\circ}(x)\}dt_1, \label{t11}\\
T_{12}=-\frac{\sigma^2}{\sqrt{nb_n}}\int_0^t\sum_{(j,k)\in I_n}\mathbb{E}\{U_{jj}*U_{kk}(t_1)e_n^{\circ}(x)\}dt_1\label{t12},\\
T_{13}=-\frac{2\sigma^2}{n}\int_0^t\sum_{(j,k)\in I_n}\mathbb{E}\{U_{jk}(t_1)xe_n(x)\varphi_{jk}'(M)\}dt_1. \label{t13}
\end{eqnarray}
It follows from the Cauchy-Schwarz inequality and $|e_n(x)|\leq 1$ that
\begin{equation}
|T_{11}|\leq\frac{\sigma^2}{\sqrt{nb_n}}\int_0^t\sum_{k=1}^nVar^{1/2}\left\{\sum_{j:{(j,k)\in I_n}}(U_{jk}*U_{jk})(t_1)\right\}dt_1
\label{t11b}.
\end{equation}
Let us fix $k$ and define
\begin{equation}
U^{(k)}(t):=(U_{jl}^{(k)}(t))_{j,l=1,...,n},
\end{equation}
where
\begin{equation}
U_{jl}^{(k)}(t)=\left\{\begin{array}{cc}U_{jl}(t)&\mbox{ if } {(j,k)\in I_n,}\\ 0& \mbox{ otherwise. }\end{array}\right. 
\end{equation}
Then
\begin{equation}
\|U^{(k)}(t)\|\leq 1.
\end{equation}
By the Poincar\'e Inequality (\ref{ppi}), (\ref{beta}), (\ref{dif}), and the Cauchy-Schwarz Inequality, we have
\begin{eqnarray}
& & Var\{\sum_{j:{(j,k)\in I_n}}U_{jk}*U_{jk}(t_1)\} \leq \frac{4}{mb_n}\sum_{(p,s)\in I_n^+}\mathbb{E}\{|\sum_{j:{(j,k)\in I_n}}U_{jp}*
U_{ks}*U_{jk}(t_1)+U_{js}*U_{kp}*U_{jk}(t_1)|^2\}\nonumber\\
& & \leq \frac{8t_1^2}{mb_n}\mathbb{E}\{\int_0^{t_1}\int_0^{t_2}\sum_{p=1}^n|(U(t_1-t_2)U^{(k)}(t_3))_{pk}|^2\sum_{s=1}^n|U_{ks}(t_2-t_3)|^2dt_3dt_2\}.
\end{eqnarray}
It follows from (\ref{Uprop}) and
\begin{equation}
\sum_{p=1}^n\left|(U(t_1-t_2)U^{(k)}(t_3))_{pk}\right|^2=\|U(t_1-t_2)U^{(k)}(t_3)e_k\|^2=\|U^{(k)}(t_3)e_k\|^2\leq\|e_k\|^2=1\label{bound1}
\end{equation}
that we have
\begin{equation}
Var\{\sum_{j:{(j,k)\in I_n}}U_{jk}*U_{jk}(t_1)\}\leq\frac{4t_1^4}{mb_n}.
\end{equation}
Hence, 
\begin{equation}
|T_{11}|\leq\frac{2\sqrt{n}\sigma^2t^3}{3\sqrt{m}b_n}.
\label{T11bound}
\end{equation}
Recall that $n/b_n^2\to 0$, so $T_{11}\to0$ as $n\to\infty$ if $t$ is bounded.

Now, we turn out attention to (\ref{t12}). We write $T_{12}$ as follows
\begin{equation}
\label{T12}
T_{12}=-\frac{2\sigma^2}{\sqrt{nb_n}}\int_0^t\int_0^{t_1}\sum_{k=1}^n\left\{\mathbb{E}\{U_{kk}(t_1-t_2)\}
\sum_{j:{(j,k)\in I_n}}\mathbb{E}\{U_{jj}(t_2)e_n^{\circ}(x)\}\right\}dt_2dt_1+T_{12}',
\end{equation}
where
\begin{equation}
T_{12}'=-\frac{\sigma^2}{\sqrt{nb_n}}\int_0^t\int_0^{t_1}\sum_{(j,k)\in I_n}
\mathbb{E}\{U_{kk}^{\circ}(t_1-t_2)U_{jj}^{\circ}(t_2)e_n^{\circ}(x)\}dt_2dt_1\label{t12'}.
\end{equation}
Since $\mathbb{E}\{U_{kk}(t)\}$ and $\mathbb{E}\{U_{kk}(t)e_n^{\circ}(x)\}$ are $k$-independent, $\mathbb{E}\{U_{kk}(t)\}=\bar{v}_n(t)$, and 
\begin{equation}
\sum_{j:{(j,k)\in I_n}}\mathbb{E}\{U_{jj}(t_2)e_n^{\circ}(x)\}=\frac{2b_n+1}{n}\mathbb{E}\{u(t_2)e_n^{\circ}(x)\}=\frac{2b_n+1}{\sqrt{nb_n}}Y_n(x,t).
\end{equation}
Thus, the first term in (\ref{T12}) can be written as
\begin{equation}
-\frac{2(2b_n+1)\sigma^2}{b_n}\int_0^t\int_0^{t_1}\bar{v}_n(t_1-t_2)Y_n(x,t_2)dt_2dt_1.
\end{equation}
We are left to bound $T_{12}'$. The Cauchy-Schwarz inequality and $|e^{\circ}_n(x)|\leq2$ imply
\begin{equation}
|T_{12}'|\leq\frac{2\sigma^2}{\sqrt{nb_n}}\int_0^t\int_0^{t_1}\sum_{k=1}^nVar^{1/2}\{U_{kk}(t_2)\}
Var^{1/2}\{\sum_{j:{(j,k)\in I_n}}U_{jj}(t_1-t_2)\}dt_2dt_1.
\end{equation}
Applying (\ref{ppi}), (\ref{dif}), (\ref{beta}) and the Cauchy-Schwarz inequality, we get
\begin{eqnarray}
Var\{U_{jk}(t)\}&\leq&\frac{4}{mb_n}\sum_{(p,s)\in I_n^+}\mathbb{E}\{|U_{jp}*U_{ks}(t)|^2\}
\leq\frac{4t}{mb_n}\mathbb{E}\{\int_0^{t}\sum_{(p,s)\in I_n^+}|U_{jp}(t_1)U_{ks}(t-t_1)|^2dt_1\}\nonumber\\
&\leq&\frac{4t}{mb_n}\mathbb{E}\{\int_0^{t}\sum_{p=1}^n|U_{jp}(t_1)|^2\sum_{s=1}^n|U_{ks}(t-t_1)|^2dt_1\}=\frac{4t^2}{mb_n},
\label{vukk}
\end{eqnarray}
and for fixed $k$,
\begin{eqnarray}
Var\{\sum_{j:(j,k)\in I_n}U_{jj}(t)\}&\leq&\frac{4}{mb_n}\sum_{(p,s)\in I_n^+}\mathbb{E}\{|\sum_{j:{(j,k)\in I_n}}U_{jp}*U_{js}(t)|^2\}\nonumber\\
&=&\frac{4}{mb_n}\sum_{(p,s)\in I_n^+}\mathbb{E}\{|\int_0^{t}\sum_{j:{(j,k)\in I_n}}U_{jp}(t_1)U_{js}(t-t_1)dt_1|^2\}\nonumber\\
&=&\frac{4}{mb_n}\sum_{(p,s)\in I_n^+}\mathbb{E}\{|\int_0^{t}(U(t_1)U^{(k)}(t-t_1))_{ps}dt_1|^2\}\nonumber\\
&\leq&\frac{4t}{mb_n}\mathbb{E}\{\int_0^t\sum_{p,s=1}^n|(U(t_1)U^{(k)}(t-t_1))_{ps}|^2dt_1\}\nonumber\\
&=&\frac{4t}{mb_n}\mathbb{E}\{\int_0^tTr U(t_1)\*U^{(k)}(t-t_1)\*U(t_1)^*\*U^{(k)}(t-t_1)^*\*dt_1\}\nonumber\\
&=&\frac{4t}{mb_n}\mathbb{E}\{\int_0^tTr U^{(k)}(t-t_1)\*U^{(k)}(t-t_1)^*\*dt_1\}.
\end{eqnarray}
Since
\begin{equation}
Tr(U^{(k)}(t)U^{(k)}(t)^*)\leq (2b_n+1)\|U^{(k)}(t)U^{(k)}(t)^*\|\leq (2b_n+1)\|U^{(k)}(t)\|^2\leq 3\*b_n,
\end{equation}
we conclude that
\begin{equation}
Var\left\{\sum_{j:{(j,k)\in I_n}}U_{jj}(t)\right\}\leq\frac{12t^2}{m}.
\end{equation}
Therefore,
\begin{equation}
|T_{12}'|\leq\frac{2\sigma^2}{\sqrt{nb_n}}\int_0^t\int_0^{t_1}n\sqrt{\frac{4 t_2^2}{mb_n}}
\sqrt{\frac{12(t_1-t_2)^2}{m}}dt_2dt_1=\frac{Const\sigma^2t^4\sqrt{n}}{mb_n}.
\label{T12'bound}
\end{equation}
Now, we turn our attention to $T_{13}$. We can rewrite (\ref{t13}) in the following form
\begin{equation}
T_{13}=xZ_n(x)A_n(t)+T_{13}',
\label{T13}
\end{equation}
where $Z_n(x)$ is given by (\ref{Zn}), $A_n(t)$ is defined in (\ref{An}), 
and
\begin{equation}
T_{13}'=-\frac{2i\sigma^2x}{n}\int_0^t
\sum_{(j,k)\in I_n}\mathbb{E}\{U_{jk}(t_1)e_n^{\circ}(x)\int_{-\infty}^{\infty}t_2U_{jk}(t_2)\hat{\varphi}(t_2)dt_2\}dt_1.
\end{equation}
Then
\begin{equation}
|T_{13}'|\leq\frac{2\sigma^2x}{n}\int_0^t\int_{-\infty}^{\infty}\sum_{k=1}^nVar^{1/2}
\left\{\sum_{j:(j,k)\in I_n}U_{jk}(t_1)U_{jk}(t_2)\right\}|t_2||\hat{\varphi}(t_2)|dt_2dt_1.
\end{equation}
Let us fix $k.$  The Poincar\'e Inequality (\ref{ppi}), together with (\ref{dif}) and (\ref{beta}) imply
\begin{eqnarray}
\lefteqn{Var\left\{\sum_{j:(j,k)\in I_n}U_{jk}(t_1)U_{jk}(t_2)\right\}}\nonumber\\
&\leq&\frac{1}{mb_n}\sum_{(p,s)\in I_n^+}
\mathbb{E}\left\{\left|
\sum_{j:(j,k)\in I_n}[U_{jp}*U_{ks}(t_1)+U_{js}*U_{kp}(t_1)]U_{jk}(t_2)+U_{jk}(t_1)[U_{jp}*U_{ks}(t_2)+U_{js}*U_{kp}(t_2)]\right|^2\right\}\nonumber\\
&\leq&\frac{4}{mb_n}\sum_{(p,s)\in I_n^+}\mathbb{E}\left\{|\sum_{j:(j,k)\in I_n}U_{jp}*U_{ks}(t_1)U_{jk}(t_2)|^2+
|\sum_{j:(j,k)\in I_n}U_{js}*U_{kp}(t_1)U_{jk}(t_2)|^2\right.\nonumber\\
&&\left.+|\sum_{j:(j,k)\in I_n}U_{jk}(t_1)U_{jp}*U_{ks}(t_2)|^2+|\sum_{j:(j,k)\in I_n}U_{jk}(t_1)U_{js}*U_{kp}(t_2))|^2\right\}\nonumber\\
&\leq&\frac{8}{mb_n}\sum_{p,s=1}^n
\mathbb{E}\left\{|\sum_{j:(j,k)\in I_n}U_{jp}*U_{ks}(t_1)U_{jk}(t_2)|^2+|\sum_{j:(j,k)\in I_n}U_{jk}(t_1)U_{jp}*U_{ks}(t_2)|^2\right\}.
\end{eqnarray}
Note that the Cauchy-Schwarz inequality gives us
\begin{eqnarray}
&&|\sum_{j:(j,k)\in I_n}U_{jp}*U_{ks}(t_1)U_{jk}(t_2)|^2=|\int_0^{t_1}\sum_{j:(j,k)\in I_n}U_{jp}(t_3)U_{ks}(t_1-t_3)U_{jk}(t_2)dt_3|^2
\nonumber\\
&&=|\int_0^{t_1}(U(t_3)U^{(k)}(t_2))_{pk}U_{ks}(t_1-t_3)dt_3|^2\leq t_1\int_0^{t_1}|(U(t_3)U^{(k)}(t_2))_{pk}U_{ks}(t_1-t_3)|^2dt_3.
\end{eqnarray}
Using (\ref{Uprop}) and (\ref{bound1}), we obtain
\begin{equation}
\sum_{p,s=1}^n|\sum_{j:(j,k)\in I_n}U_{jp}*U_{ks}(t_1)U_{jk}(t_2)|^2\leq t_1\int_0^{t_1}\sum_{p=1}^n|(U(t_3)U^{(k)}(t_2))_{pk}|^2
\sum_{s=1}^n|U_{ks}(t_1-t_3)|^2dt_3\leq \frac{t_1^2}{2}.
\end{equation}
Hence,
\begin{equation}
Var\{\sum_{j:(j,k)\in I_n}U_{jk}(t_1)U_{jk}(t_2)\}\leq\frac{4(t_1^2+t_2^2)}{mb_n}.
\end{equation}
Therefore,
\begin{eqnarray}
|T_{13}'|&\leq&2\sigma^2|x|\int_0^t\int_{-\infty}^{\infty}|t_2|\sqrt{\frac{4(t_1^2+t_2^2)}{mb_n}}|\hat{\varphi}(t_2)|dt_2dt_1\nonumber\\
&=&2\sigma^2|x|\int_0^t\left[\int_{|t_2|\leq t_1}|t_2|\sqrt{\frac{4(t_1^2+t_2^2)}{mb_n}}|\hat{\varphi}(t_2)|dt_2+\int_{|t_2|>t_1}|t_2|
\sqrt{\frac{4(t_1^2+t_2^2)}{mb_n}}|\hat{\varphi}(t_2)|dt_2\right]dt_1\nonumber\\
&\leq&\frac{4\sqrt{2}\sigma^2|x|}{\sqrt{mb_n}}\int_0^t\left[t_1^2\int_{|t_2|\leq t_1}|\hat{\varphi}(t_2)|dt_2+
\int_{|t_2|>t_1}|t_2|^2|\hat{\varphi}(t_2)|dt_2\right]dt_1\nonumber\\
&\leq&\frac{\sigma^2|x|}{\sqrt{mb_n}}C_3(t).
\label{T13'bound}
\end{eqnarray}
Combining the bounds obtained in this subsection, we get
\begin{equation}
T_1= -\frac{2(2b_n+1)\sigma^2}{b_n}\int_0^t\int_0^{t_1}\bar{v}_n(t_1-t_2)Y_n(x,t_2)dt_2dt_1 + xZ_n(x)A_n(t)+\varepsilon_n(x,t).
\end{equation}
Since $b_n/n\to 0, \ n/b_n^2\to 0$, using (\ref{T11bound}), (\ref{T12'bound}), and (\ref{T13'bound}), we have that 
$\varepsilon_n(x,t)=T_{11}+T_{12}'+T_{13}'\to 0$ on any bounded subset of $\{(x,t),t\geq0\}.$
Proposition \ref{T1} is proven. \end{proof}

\subsection{Estimate of $T_2$}
\label{TT2}
The main result of this subsection is the following proposition.
\begin{prop}
\label{t2}
Let $T_2$ be defined as in (\ref{t1}) with $l=2$.
Then $T_2$ converges to zero as $n\to \infty$ uniformly on any bounded subset of $\{t\geq0,x\in\mathbb{R}\}$.
\end{prop}

\begin{proof}
\begin{equation}
T_2=\frac{i\kappa_3}{2\sqrt{n}b_n}\int_0^t\sum_{(j,k)\in I_n}\mathbb{E}\{D_{jk}^2(U_{jk}(t_1)e_n^{\circ}(x))\}dt_1+
\frac{i(\kappa_3'-\kappa_3)}{2\sqrt{n}b_n}\int_0^t\sum_{j=1}^n\mathbb{E}\{D_{jj}^2(U_{jj}(t_1)e_n^{\circ}(x))\}dt_1.
\end{equation}
By Lemma \ref{difb}, the second term in $T_2$ is bounded by $\frac{\sqrt{n}}{b_n}C_3((b_n/n)^{1/2}x,t).$ 
As for  the first term in $T_2$, it can be written as the sum of $T_{21}$ and $T_{22},$  where
\begin{eqnarray}
T_{21}&=&-\frac{i\kappa_3}{\sqrt{n}b_n}\int_0^t\sum_{(j,k)\in I_n}\mathbb{E}\{\beta_{jk}^2(U_{jk}*U_{jk}*U_{jk})(t_1)e_n^{\circ}(x)\}dt_1\nonumber\\
&&+\frac{2\kappa_3}{n\sqrt{b_n}}\int_0^t\sum_{(j,k)\in I_n}
\mathbb{E}\{\beta_{jk}^2(U_{jk}*U_{jk})(t_1)xe_n(x)\int_{-\infty}^{\infty}t_2U_{jk}(t_2)\hat{\varphi}(t_2)dt_2\}dt_1\nonumber\\
&&+\frac{2i\kappa_3}{n^{3/2}}\int_0^t\sum_{(j,k)\in I_n}
\mathbb{E}\{\beta_{jk}^2x^2U_{jk}(t_1)e_n(x)(\int_{-\infty}^{\infty} t_2U_{jk}(t_2)\hat{\varphi}(t_2)dt_2)^2\}dt_1\nonumber\\
&&+\frac{\kappa_3}{n\sqrt{b_n}}\int_0^t\sum_{(j,k)\in I_n}
\mathbb{E}\{\beta_{jk}^2xU_{jk}(t_1)e_n(x)\int_{-\infty}^{\infty} t_2(U_{jk}*U_{jk})(t_2)\hat{\varphi}(t_2)dt_2\}dt_1,
\end{eqnarray}
\begin{eqnarray}
T_{22}&=&-\frac{i\kappa_3}{\sqrt{n}b_n}\int_0^t\sum_{(j,k)\in I_n}\mathbb{E}\{\beta_{jk}^2(3U_{jj}*U_{jk}*U_{kk})(t_1)e_n^{\circ}(x)\}dt_1\nonumber\\
&&+\frac{2\kappa_3}{n\sqrt{b_n}}\int_0^t\sum_{(j,k)\in I_n}\mathbb{E}\{\beta_{jk}^2(U_{jj}*U_{kk})(t_1)xe_n(x)
\int_{-\infty}^{\infty}t_2U_{jk}(t_2)\hat{\varphi}(t_2)dt_2\}dt_1\nonumber\\
&&+\frac{\kappa_3}{n\sqrt{b_n}}\int_0^t\sum_{(j,k)\in I_n}\mathbb{E}\{\beta_{jk}^2xU_{jk}(t_1)e_n(x)
\int_{-\infty}^{\infty} t_2(U_{jj}*U_{kk})(t_2)\hat{\varphi}(t_2)dt_2\}dt_1.
\end{eqnarray}
Since for any $s_1,s_2,s_3\in \mathbb{R},$ one has
\begin{equation}
|\sum_{(j,k)\in I_n}U_{jk}(s_1)U_{jk}(s_2)U_{jk}(s_3)|\leq \sum_{j,k=1}^n|U_{jk}(s_1)U_{jk}(s_2)|\leq\left(\sum_{j,k=1}^n|U_{jk}(s_1)|^2\right)^{1/2}
\left(\sum_{j,k=1}^n|U_{jk}(s_2)|^2\right)^{1/2}=n.
\end{equation}
We have
\begin{equation}
|T_{21}|\leq \sqrt{n}/b_nC_2(\sqrt{b_n/n}x)t.
\label{T21bound}
\end{equation}
To estimate $T_{22}$, we note that
\begin{equation}
\left|\mathbb{E}\left\{\sum_{(j,k)\in I_n}\beta_{jk}^2U_{jj}(s_1)U_{jk}(s_2)U_{kk}(s_3)e_n^{\circ}(x)\right\}\right|
\leq Var^{1/2}\left\{\sum_{(j,k)\in I_n}\beta_{jk}^2U_{jj}(s_1)U_{jk}(s_2)U_{kk}(s_3)\right\}.
\end{equation}
Using the Poincar\'e inequality, we obtain an upper bound
\begin{eqnarray}
\lefteqn{Var\left\{\sum_{(j,k)\in I_n}\beta_{jk}^2U_{jj}(s_1)U_{jk}(s_2)U_{kk}(s_3)\right\}}\nonumber\\
&\leq&\frac{12}{mb_n}\sum_{(p,q)\in I_n^+}\mathbb{E}\left\{|\sum_{(j,k)\in I_n}U_{jp}*U_{jq}(s_1)U_{jk}(s_2)U_{kk}(s_3)|^2\right.+
|\sum_{(j,k)\in I_n}U_{jj}(s_1)U_{jp}*U_{kq}(s_2)U_{kk}(s_3)|^2\nonumber\\
&&+\left.|\sum_{(j,k)\in I_n}U_{jj}(s_1)U_{jq}*U_{kp}(s_2)U_{kk}(s_3)|^2+|\sum_{(j,k)\in I_n}U_{jj}(s_1)U_{jk}(s_2)U_{kp}*U_{kq}(s_3)|^2\right\}.
\end{eqnarray}
Let $\alpha_j=\sum_{k:(j,k)\in I_n}U_{jk}(s_2)U_{kk}(s_3)$ and $ D=diag\{\alpha_1,...,\alpha_n\}$. Then $|\alpha_j|\leq\sqrt{2b_n+1}$, and
\begin{eqnarray}
&&\sum_{(p,q)\in I_n^+}|\sum_{(j,k)\in I_n}U_{jp}*U_{jq}(s_1)U_{jk}(s_2)U_{kk}(s_3)|^2
\leq\sum_{(p,q)\in I_n^+}|s_1|\*\int_0^{s_1}|\sum_{(j,k)\in I_n}U_{jp}(t)U_{jq}(s_1-t)U_{jk}(s_2)U_{kk}(s_3)|^2dt\nonumber\\
&&=|s_1|\*\int_0^{s_1}\sum_{(p,q)\in I_n^+}|\sum_{j=1}^nU_{jp}(t)U_{jq}(s_1-t)\alpha_j|^2dt=|s_1|\*\int_0^{s_1}\*
\sum_{(p,q)\in I_n^+}|(U(t)DU(s_1-t))_{pq}|^2dt\nonumber\\
&&\leq |s_1|\int_0^{s_1}\*n\|U(t)DU(s_1-t)\|^2dt\leq s_1^2 \*n(2b_n+1).
\end{eqnarray}
Now let $D_1=diag\{U_{11}(s_1),...,U_{nn}(s_1)\}, \ D_2=diag\{U_{11}(s_3),...,U_{nn}(s_3)\},$ and $B=(B_{jk})_{j,k=1}^n$ be a $0$-$1$ 
band matrix, such that $B_{jk}=1_{(j,k)\in I_n}$. Then 
\begin{eqnarray}
& & \sum_{(p,q)\in I_n^+}|\sum_{(j,k)\in I_n}U_{jj}(s_1)U_{jp}*U_{kq}(s_2)U_{kk}(s_3)|^2 
\leq\sum_{(p,q)\in I_n^+} |s_2|\*\int_0^{s_2}|\sum_{(j,k)\in I_n}U_{jj}(s_1)U_{jp}(t)U_{kq}(s_2-t)U_{kk}(s_3)|^2dt\nonumber\\
& &=|s_2|\*\int_0^{s_2}\*\sum_{(p,q)\in I_n^+}|(U(t)D_1BD_2U(s_2-t))_{pq}|^2dt \leq s_2^2 \*n\|B\|^2 = s_2^2\*n(2b_n+1)^2. 
\end{eqnarray}
Hence,
\begin{eqnarray}
& & Var\{\sum_{(j,k)\in I_n}\beta_{jk}^2U_{jj}(s_1)U_{jk}(s_2)U_{kk}(s_3)\}\leq \frac{12n(2b_n+1)}{mb_n}(s_1^2+s_3^2+2\*s_2^2\*(2b_n+1)) \\
& & \leq n\*b_n\*C_2(s_1,s_2,s_3),\\
& & |\mathbb{E}\{\sum_{(j,k)\in I_n}\beta_{jk}^2U_{jj}(s_1)U_{jk}(s_2)U_{kk}(s_3)e_n^{\circ}(x)\}|\leq\sqrt{n\*b_n}C_2^{1/2}(s_1,s_2,s_3).
\end{eqnarray}
For the last two terms in $T_{22}$, taking into account that
\begin{equation}
\mathbb{E}\{U_{jj}(s_1)U_{jk}(s_2)U_{kk}(s_3)e_n(x)\}=\mathbb{E}\{U_{jj}(s_1)U_{jk}(s_2)U_{kk}(s_3)e_n^{\circ}(x)\}+\mathbb{E}\{U_{jj}(s_1)U_{jk}(s_2)
U_{kk}(s_3)\}
\mathbb{E}\{e_n(x)\},
\end{equation}
we have
\begin{equation}
|\sum_{(j,k)\in I_n}\beta_{jk}^2\mathbb{E}\{U_{jj}(s_1)U_{jk}(s_2)U_{kk}(s_3)e_n(x)\}|\leq \sqrt{nb_n}C_2^{1/2}(s_1,s_2,s_3)
+n/4+\sum_{(j,k)\in I_n,j\neq k}|\mathbb{E}\{U_{jj}(s_1)U_{jk}(s_2)U_{kk}(s_3)\}|.
\end{equation}
For $j\neq k,$ we can write
\begin{eqnarray}
& & \mathbb{E}\{U_{jj}(s_1)U_{jk}(s_2)U_{kk}(s_3)\}=\mathbb{E}\{U_{jj}^{\circ}(s_1)U_{jk}(s_2)U_{kk}^{\circ}(s_3)\}+\mathbb{E}\{U_{jj}(s_1)\}
\mathbb{E}\{U_{jk}^{\circ}(s_2)U_{kk}^{\circ}(s_3)\}\nonumber\\
& &+\mathbb{E}\{U_{jj}^{\circ}(s_1)U_{jk}^{\circ}(s_2)\}\mathbb{E}\{U_{kk}(s_3)\}+\mathbb{E}\{U_{jj}(s_1)\}\mathbb{E}\{U_{jk}(s_2)\}
\mathbb{E}\{U_{kk}(s_3)\}\label{ejjjkkk}.
\end{eqnarray}
So the left hand side of (\ref{ejjjkkk}) is bounded by
\begin{equation}
(Var\{U_{jj}(s_1)\}Var\{U_{kk}(s_2)\})^{1/2}+(Var\{U_{kk}(s_3)\}Var\{U_{jk}(s_2)\})^{1/2}+(Var\{U_{jj}(s_1)\}Var\{U_{jk}(s_2)\})^{1/2}+
|\mathbb{E}\{U_{jk}(s_2)\}|.
\end{equation}
By (\ref{vukk}) it is bounded from above by
\begin{equation}
\frac{4(|s_1\*s_2|+|s_1\*s_3|+|s_2\*s_3|)}{mb_n}+|\mathbb{E}\{U_{jk}(s_2)\}|.
\end{equation}
To bound the second term in the last expression, we use the following auxiliary proposition.
\begin{prop}
\label{propprop}
Let $M=W/\sqrt{b_n}$ be a real symmetric band random matrix defined as Theorem \ref{thm-Poin}
and \\
$U(t)=e^{itM}.$ Then 
\begin{equation}
\label{nem}
\sup _{j\neq k}  |\mathbb{E}\{U_{jk}(t)\}|=O\left(\frac{1+t^6}{b_n}\right).
\end{equation}
\end{prop}
The proof of Proposition \ref{propprop} is given in Appendix C.
Assuming (\ref{nem}), we conclude that
\begin{equation}
|\mathbb{E}\{U_{jj}(s_1)U_{jk}(s_2)U_{kk}(s_3)\}|\leq\frac{4(|s_1\*s_2|+|s_1\*s_3|+|s_2\*s_3|)}{mb_n}+O(\frac{1+s_2^6}{b_n}).
\end{equation}
Therefore,
\begin{equation}
|\sum_{(j,k)\in I_n}\beta_{jk}^2\mathbb{E}\{U_{jj}(s_1)U_{jk}(s_2)U_{kk}(s_3)e_n(x)\}|\leq \sqrt{nb_n}C_2^{1/2}(s_1,s_2,s_3)
+n/4+\frac{8\*n(|s_1\*s_2|+|s_1\*s_3|+|s_2\*s_3|)}{m}+ O(s_2^6\*n),
\end{equation}
and
\begin{equation}
|T_{22}|\leq\frac{\sqrt{n}}{b_n}C_4(t)+\frac{|x|}{b_n^{\varepsilon/6}}\*C_4(t)\label{T22bound}.
\end{equation}
This and (\ref{T21bound}) imply that $T_2$ converges to zero uniformly on any bounded subset of $\{t\geq0,x\in\mathbb{R}\}$.
Proposition \ref{t2} is proven. \end{proof}

\subsection{Estimate of $T_3$.}
\label{TT3}

Finally, let us consider $T_3.$
The main result of this subsection is the following bound.

\begin{prop}
\label{T3}
Let $T_3$ be defined as in (\ref{t1}) (with $l=3).$  Then
\begin{equation}
\label{davis}
T_3(x,t)= 
2\*i\*\kappa_4xZ_n(x)\int_0^t\bar{v}_n*\bar{v}_n(t_1)dt_1\int_{-\infty}^{\infty}t_2\bar{v}_n*\bar{v}_n(t_2)\hat{\varphi}(t_2)dt_2+\epsilon_n(x,t),
\end{equation}
where 
$\epsilon_n(x,t) \to 0$ as $n\to \infty$ uniformly on any bounded subset of $\{(x,t),t\geq0\}.$
\end{prop}

\begin{proof}
One has
\begin{equation}
\label{eqn:T3}
T_3(x,t)=\frac{i\kappa_4}{6\sqrt{nb_n^3}}\int_0^t\sum_{(j,k)\in I_n}\mathbb{E}\{D_{jk}^3(U_{jk}(t_1)e_n^{\circ}(x))\}dt_1+T_{3}'(x,t),
\end{equation}
where the term $T_3'$ comes from the summation over $j=k$ and corresponds to the fact that the marginal distribution of the diagonal entries is different 
from the marginal distribution of the off-diagonal entries.
It follows form Lemma \ref{difb} that $T_3'$ can be bounded as
\begin{equation}
\label{T32bound}
|T_3'(x,t)|\leq\sqrt{\frac{n}{b_n^3}} C_4((b_n/n)^{1/2}x,t).
\end{equation}
By (\ref{dif}), the first term in (\ref{eqn:T3}) can be written as the sum of 
\begin{eqnarray}
\label{komar1}
&&T_{31}(x,t)=\frac{\kappa_4}{\sqrt{nb_n^3}}\int_0^t\sum_{(j,k)\in I_n}\mathbb{E}\{\beta_{jk}^3(U_{jj}*U_{jj}*U_{kk}*U_{kk})(t_1)e_n^{\circ}(x)\}dt_1 
\\
\label{komar2}
&&+\frac{i\kappa_4}{nb_n}\int_0^t\sum_{(j,k)\in I_n}
\mathbb{E}\{\beta_{jk}^3(U_{jj}*U_{kk})(t_1)xe_n(x)\int_{-\infty}^{\infty}t_2(U_{jj}*U_{kk})(t_2)\hat{\varphi}(t_2)dt_2\}dt_1,
\end{eqnarray}
and
\begin{eqnarray}
&&T_{32}(x,t)=\frac{\kappa_4}{\sqrt{nb_n^3}}\int_0^t\sum_{(j,k)\in I_n}
\mathbb{E}\{\beta_{jk}^3(6U_{jj}*U_{jk}*U_{jk}*U_{kk}+U_{jk}*U_{jk}*U_{jk}*U_{jk})(t_1)e_n^{\circ}(x)\}dt_1\nonumber\\
&&+\frac{i\kappa_4}{nb_n}\int_0^t\sum_{(j,k)\in I_n}
\mathbb{E}\{\beta_{jk}^3(6U_{jj}*U_{jk}*U_{kk}+2U_{jk}*U_{jk}*U_{jk})(t_1)xe_n(x)\int_{-\infty}^{\infty}t_2U_{jk}(t_2)
\hat{\varphi}(t_2)dt_2\}dt_1\nonumber\\
&&+\frac{i\kappa_4}{nb_n}\int_0^t\sum_{(j,k)\in I_n}\mathbb{E}\{\beta_{jk}^3(U_{jj}*U_{kk}+U_{jk}*U_{jk})(t_1)xe_n(x)
\int_{-\infty}^{\infty}t_2(U_{jk}*U_{jk})(t_2)\hat{\varphi}(t_2)dt_2\}dt_1\nonumber\\
&&+\frac{i\kappa_4}{nb_n}\int_0^t\sum_{(j,k)\in I_n}\mathbb{E}\{\beta_{jk}^3(U_{jk}*U_{jk})(t_1)xe_n(x)
\int_{-\infty}^{\infty}t_2(U_{jj}*U_{kk}+U_{jk}*U_{jk})(t_2)\hat{\varphi}(t_2)dt_2\}dt_1\nonumber\\
&&+\frac{i\kappa_4}{3nb_n}\int_0^t\sum_{(j,k)\in I_n}\mathbb{E}\{\beta_{jk}^3U_{jk}(t_1)xe_n(x)
\int_{-\infty}^{\infty}t_2(6U_{jj}*U_{jk}*U_{kk}+2U_{jk}*U_{jk}*U_{jk})(t_2)\hat{\varphi}(t_2)dt_2\}dt_1\nonumber\\
&&-\frac{2\kappa_4}{\sqrt{n^3b_n}}\int_0^t\sum_{(j,k)\in I_n}
\mathbb{E}\{\beta_{jk}^3(U_{jj}*U_{kk}+U_{jk}*U_{jk})(t_1)x^2e_n(x)(\int_{-\infty}^{\infty}t_2U_{jk}(t_2)\hat{\varphi}(t_2)dt_2)^2\}dt_1\nonumber\\
&&-\frac{2\kappa_4}{\sqrt{n^3b_n}}\int_0^t\sum_{(j,k)\in I_n}\mathbb{E}\{\beta_{jk}^3U_{jk}(t_1)x^2e_n(x)
\int_{-\infty}^{\infty}t_2U_{jk}(t_2)\hat{\varphi}(t_2)dt_2\int_{-\infty}^{\infty}t_3(U_{jj}*U_{kk}+U_{jk}*U_{jk})(t_3)
\hat{\varphi}(t_3)dt_3\}dt_1\nonumber\\
&&-\frac{4i\kappa_4}{3n^2}\int_0^t\sum_{(j,k)\in I_n}
\mathbb{E}\{\beta_{jk}^3U_{jk}(t_1)x^3e_n(x)(\int_{-\infty}^{\infty}t_2U_{jk}(t_2)\hat{\varphi}(t_2)dt_2)^3\}dt_1.
\label{eqn:T32}
\end{eqnarray}
Thus, $T_3=T_{31}+T_{32}+T_3'.\ $  Since we have already bounded $T_3'$ in (\ref{T32bound}), we are left with estimating the first two terms in the sum.
There are two types of sums over $(j,k)\in I_n$ in (\ref{eqn:T32}), namely the first one corresponding to $U_{jj}\*U_{jk}\*U_{jk}\*U_{kk}$ and 
the second one to $U_{jk}\*U_{jk}\*U_{jk}\*U_{jk}$. Define
\begin{eqnarray}
J_{1}(s_1,s_2,s_3,s_4)=\sum_{(j,k)\in I_n}U_{jj}(s_1)U_{jk}(s_2)U_{jk}(s_3)U_{kk}(s_4),\\
J_{2}(s_1,s_2,s_3,s_4)=\sum_{(j,k)\in I_n}U_{jk}(s_1)U_{jk}(s_2)U_{jk}(s_3)U_{jk}(s_4).
\end{eqnarray}
We note that
\begin{eqnarray}
|J_{1}|\leq\sum_{(j,k)\in I_n}|U_{jk}(s_2)U_{jk}(s_3)|\leq\left\{\sum_{j,k=1}^n|U_{jk}(s_2)|^2\sum_{j,k=1}^n|U_{jk}(s_3)|^2\right\}^{1/2}=n.
\end{eqnarray}
Similarly,
\begin{equation}
|J_{2}|\leq n.
\end{equation}
It follows from the last two inequalities and (\ref{ineq:fourier}) that
\begin{equation}
\label{t32xt}
|T_{32}(x,t)|\leq\sqrt{\frac{n}{b_n^3}}C_4((b_n/n)^{1/2}x,t). 
\end{equation}
Now, we estimate $T_{31}.$   Recall that $T_{31}$ is defined in
(\ref{komar1}-\ref{komar2}). Let us denote
\begin{equation}
\label{shah1}
v_n(s_1,s_2,s_3,s_4)=(nb_n)^{-1}\sum_{(j,k)\in I_n}U_{jj}(s_1)U_{jj}(s_2)U_{kk}(s_3)U_{kk}(s_4),
\end{equation}
and 
\begin{equation}
\label{shah2}
\bar{v}_n(s_1,s_2,s_3,s_4):=\mathbb{E}\{v_n(s_1,s_2,s_3,s_4)\}.
\end{equation}
The rest of the proof of Proposition \ref{T3} follows from the next two lemmas.

\begin{lem}
\begin{equation}
\label{shah3}
\bar{v}_n(s_1,s_2,s_3,s_4)=2\bar{v}_n(s_1)\bar{v}_n(s_2)\bar{v}_n(s_3)\bar{v}_n(s_4)+h(s_1,s_2,s_3,s_4),
\end{equation}
where $\bar{v}_n(\cdot)$ is given by (\ref{vnt}) and
\begin{equation}
\label{shah4}
|h(s_1,s_2,s_3,s_4)|\leq \frac{6(|s_1|+|s_2|+|s_3|)}{\sqrt{m\*b_n}}+\frac{1}{b_n}.
\end{equation}
\end{lem}
\begin{proof}
\begin{eqnarray}
&&\mathbb{E}\{U_{jj}(s_1)U_{jj}(s_2)U_{kk}(s_3)U_{kk}(s_4)\}\nonumber=\mathbb{E}\{U_{jj}^{\circ}(s_1)U_{jj}(s_2)U_{kk}(s_3)U_{kk}(s_4)\}+
\mathbb{E}\{U_{jj}(s_1)\}\mathbb{E}\{U_{jj}^{\circ}(s_2)U_{kk}(s_3)U_{kk}(s_4)\}\nonumber\\
&&+\mathbb{E}\{U_{jj}(s_1)\}\mathbb{E}\{U_{jj}(s_2)\}\mathbb{E}\{U_{kk}^{\circ}(s_3)U_{kk}(s_4)\}+
\mathbb{E}\{U_{jj}(s_1)\}\mathbb{E}\{U_{jj}(s_2)\}\mathbb{E}\{U_{kk}(s_3)\}\mathbb{E}\{U_{kk}(s_4)\}.
\end{eqnarray}
It follows from (\ref{vukk}) that
\begin{eqnarray}
\lefteqn{|\mathbb{E}\{U_{jj}(s_1)U_{jj}(s_2)U_{kk}(s_3)U_{kk}(s_4)\} - \mathbb{E}\{U_{jj}(s_1)\}\mathbb{E}\{U_{jj}(s_2)\}\mathbb{E}\{U_{kk}(s_3)\}
\mathbb{E}\{U_{kk}(s_4)\}|}
\nonumber\\
&&\leq Var^{1/2}\{U_{jj}(s_1)\}+Var^{1/2}\{U_{jj}(s_2)\}+Var^{1/2}\{U_{kk}(s_3)\} \leq\frac{2(|s_1|+|s_2|+|s_3|)}{\sqrt{m\*b_n}}.
\end{eqnarray}
Therefore, we obtain
\begin{equation}
\label{mytischi1}
|\bar{v}_n(s_1,s_2,s_3,s_4)-(nb_n)^{-1}\sum_{(j,k)\in I_n}\mathbb{E}\{U_{jj}(s_1)\}\mathbb{E}\{U_{jj}(s_2)\}\mathbb{E}\{U_{kk}(s_3)\}
\mathbb{E}\{U_{kk}(s_4)\}|\leq \frac{6(|s_1|+|s_2|+|s_3|)}{\sqrt{mb_n}}.
\end{equation}
In addition,
\begin{equation}
\label{mytischi2}
(nb_n)^{-1}\sum_{(j,k)\in I_n}\mathbb{E}\{U_{jj}(a)\}\mathbb{E}\{U_{jj}(b)\}\mathbb{E}\{U_{kk}(c)\}
\mathbb{E}\{U_{kk}(d)\}=(2+1/b_n)\bar{v}_n(a)\bar{v}_n(b)\bar{v}_n(c)\bar{v}_n(d)
\end{equation}
Now the lemma follows from (\ref{mytischi1}-\ref{mytischi2}) and 
$\bar{v}_n(t)\leq 1.$ 
\end{proof}
The next lemma deals with $T_{31}$ defined in (\ref{komar1}-\ref{komar2}).
\begin{lem}
\label{mmm}
\begin{equation}
T_{31}=2i\kappa_4xZ_n(x)\int_0^t\int_{-\infty}^{\infty}t_2\bar{v}_n*\bar{v}_n(t_1)\bar{v}_n*\bar{v}_n(t_2)\hat{\varphi}(t_2)dt_2dt_1+\delta_n(x,t),
\end{equation}
where
\begin{equation}
\label{dddd}
\delta_n(x,t)\to 0
\end{equation}
uniformly on any bounded subset of 
$\{(x,t),t\geq0\}.$
\end{lem}
\begin{proof}
$T_{31}$ can be written as
\begin{eqnarray}
\label{komar11}
T_{31}(x,t)&=&\frac{\kappa_4}{\sqrt{nb_n^3}}\int_0^t\sum_{(j,k)\in I_n}\mathbb{E}\{(U_{jj}*U_{jj}*U_{kk}*U_{kk})(t_1)e_n^{\circ}(x)\}dt_1 \nonumber\\
&&+\frac{i\kappa_4}{nb_n}\int_0^t\sum_{(j,k)\in I_n}
\mathbb{E}\{(U_{jj}*U_{kk})(t_1)xe_n^{\circ}(x)\int_{-\infty}^{\infty}t_2(U_{jj}*U_{kk})(t_2)\hat{\varphi}(t_2)dt_2\}dt_1\nonumber\\
&&+\frac{i\kappa_4}{nb_n}\int_0^t\sum_{(j,k)\in I_n}
\mathbb{E}\{(U_{jj}*U_{kk})(t_1)x\mathbb{E}\{e_n(x)\}\int_{-\infty}^{\infty}t_2(U_{jj}*U_{kk})(t_2)\hat{\varphi}(t_2)dt_2\}dt_1
\nonumber\\
&&+T_{31}'(x,t)\nonumber\\
&=:&T_{311}+T_{312}+T_{313}+T_{31}'(x,t).
\end{eqnarray}
where $T_{31}'$ comes from the diagonal terms, so $|T_{31}'|\leq\sqrt{\frac{n}{b_n^3}}C_4(\sqrt{b_n/n}x,t)$.
Then
\begin{eqnarray}
\label{t311}
T_{311}&=&\frac{\kappa_4\sqrt{n}}{\sqrt{b_n}}\int_0^t\int_0^{t_1}\int_0^{t_2}\int_0^{t_3}\mathbb{E}\{v_n(t_1-t_2,t_2-t_3,t_3-t_4,t_4)
e_n^{\circ}(x)\}dt_4dt_3dt_2dt_1, \\
\label{t312}
T_{312}&=&i\*\kappa_4x\*\int_0^t\int_{-\infty}^{\infty}\int_0^{t_1}\int_0^{t_2}t_2\mathbb{E}\{v_n(t_3,t_4,t_1-t_3,t_2-t_4)
e_n^{\circ}(x)\}\hat{\varphi}(t_2)dt_4 dt_3dt_2 dt_1,\\
T_{313}&=&i\*\kappa_4x\*Z_n(x)\int_0^t\int_{-\infty}^{\infty}\int_0^{t_1}\int_0^{t_2}t_2\bar{v}_n(t_3,t_4,t_1-t_3,t_2-t_4)
\hat{\varphi}(t_2)dt_4 dt_3 dt_2dt_1.
\label{t313}
\end{eqnarray}
Moreover, by Lemma \ref{mmm},
\begin{eqnarray}
 T_{313}&=&2\*i\kappa_4xZ_n(x)\int_0^t\int_{-\infty}^{\infty}
\int_0^{t_1}\int_0^{t_2}t_2 \bar{v}_n(t_3)\bar{v}_n(t_4)\bar{v}_n(t_1-t_3)\bar{v}_n(t_2-t_4)\hat{\varphi}(t_2)dt_4 dt_3 dt_2dt_1 +\tau_n(x,t) \nonumber\\
&=&2\* i\kappa_4xZ_n(x)\int_0^t\int_{-\infty}^{\infty}t_2\bar{v}_n*\bar{v}_n(t_1)\bar{v}_n*\bar{v}_n(t_2)
\hat{\varphi}(t_2)dt_2dt_1+\tau_n(x,t),
\end{eqnarray}
where
\[\tau_n(x,t)=i\kappa_4xZ_n(x)\int_0^t\int_{-\infty}^{\infty}\int_0^{t_1}\int_0^{t_2}t_2\*h(t_3,t_4,t_1-t_3,t_2-t_4)
\hat{\varphi}(t_2)dt_4 dt_3 dt_2dt_1.\]
Thus,
\begin{eqnarray}
|\tau_n(x,t)|&\leq& |\kappa_4xZ_n(x)|\int_0^t\int_{-\infty}^{\infty}
\int_0^{t_1}\left|\int_0^{t_2}|t_2\*h(t_3,t_4,t_1-t_3,t_2-t_4)|\hat{\varphi}(t_2)dt_4\right| dt_3 dt_2dt_1\nonumber\\
&\leq&\frac{|x|}{\sqrt{b_n}}C(t^4+1).
\end{eqnarray}
Since $|e_n^{\circ}(x)|\leq2$,
\begin{eqnarray}
|\mathbb{E}\{v_n(s_1, s_2, s_3, s_4)e_n^{\circ}\}|&\leq&2\mathbb{E}\{|v_n^{\circ}(s_1,s_2,s_3,s_4)|\}\leq\frac{2}{nb_n}
\sum_{(j,k)\in I_n}\mathbb{E}\{|(U_{jj}(s_1)U_{jj}(s_2)U_{kk}(s_3)U_{kk}(s_4))^{\circ}|\}\nonumber\\
&\leq&\frac{4}{nb_n}\sum_{(j,k)\in I_n}\mathbb{E}\{|(U_{jj}(s_1)U_{jj}(s_2))^{\circ}|\}+\mathbb{E}\{|(U_{kk}(s_3)U_{kk}(s_4))^{\circ}|\}.
\end{eqnarray}
Also,
\begin{eqnarray}
Var\{U_{jj}(s_1)U_{jj}(s_2)\}&\leq&\frac{4}{mb_n}
\sum_{(p,q)\in I_n^+}\mathbb{E}\{|U_{jp}*U_{jq}(s_1)U_{jj}(s_2)+U_{jj}(s_1)U_{jp}*U_{jq}(s_2)|^2\}\nonumber\\
&\leq&\frac{8}{mb_n}\sum_{(p,q)\in I_n}\mathbb{E}\{|U_{jp}*U_{jq}(s_1)U_{jj}(s_2)|^2\}+\mathbb{E}\{U_{jj}(s_1)U_{jp}*U_{jq}(s_2)|^2\},
\end{eqnarray}
and
\begin{eqnarray}
\sum_{(p,q)\in I_n}\mathbb{E}\{|U_{jp}*U_{jq}(s_1)U_{jj}(s_2)|^2\}
&=&\mathbb{E}\{\sum_{(p,q)\in I_n}|\int_0^aU_{jp}(s)U_{jq}(s_1-s)U_{jj}(s_2)ds|^2\}\nonumber\\
&\leq& |s_1|\*  \*\mathbb{E}\{\sum_{(p,q)\in I_n}\int_0^{s_1}|U_{jp}(s)U_{jq}(s_1-s)|^2ds\}\leq s_1^2.
\end{eqnarray}
Thus,
\begin{equation}
Var\{U_{jj}(s_1)U_{jj}(s_2)\}\leq\frac{8(s_1^2+s_2^2)}{mb_n},
\end{equation}
and we obtain
\begin{equation}
|\mathbb{E}\{v_n(s_1,s_2,s_3,s_4)e_n^{\circ}\}|\leq 12\*\sqrt{\frac{8(s_1^2+s_2^2)}{mb_n}}+12\sqrt{\frac{8(s_3^2+s_4^2)}{mb_n}}=\frac{C}{\sqrt{mb_n}}
(\sqrt{s_1^2+s_2^2}+\sqrt{s_3^2+s_4^2}).
\end{equation}
So 
\begin{equation}
|T_{311}|\leq\frac{C|\kappa_4|\sqrt{n}}{\sqrt{m}b_n}\int_0^t\int_0^{t_1}\int_0^{t_2}\int_0^{t_3}\sqrt{(t_1-t_2)^2+(t_2-t_3)^2}+\sqrt{(t_3-t_4)^2+t_4^2}
\*dt_4dt_3dt_2dt_1\leq\frac{C\sqrt{n}}{b_n}|t|^5,
\end{equation}
and 
\begin{equation}
|T_{312}|\leq\frac{C|\kappa_4x|}{\sqrt{mb_n}}\int_0^t\int_{-\infty}^{\infty}\int_0^{t_1}|\int_0^{t_2}|t_2|(
\sqrt{t_3^2+t_4^2}+\sqrt{(t_1-t_3)^2+(t_2-t_4)^2})|\hat{\varphi}(t_2)|dt_4| dt_3dt_2 dt_1\leq\frac{|x|}{\sqrt{b_n}}C(|t|^3+1).
\end{equation}
Since $n/b_n^2\to0$, we observe that $\delta=T_{311}+T_{312}+\tau_n+T_{31}'$ goes to zero uniformly on any bounded subset of $\{t\geq0,x\in\mathbb{R}\}$.
This finishes the proof of Lemma \ref{mmm}.
\end{proof}
Now, we are ready to finish the proof of Proposition \ref{T3}. Indeed,
\begin{eqnarray}
T_3&=&T_{31}+T_{32}+T_3'\nonumber\\
&=&2i\kappa_4xZ_n(x)\int_0^t\int_{-\infty}^{\infty}t_2\bar{v}_n*\bar{v}_n(t_1)\bar{v}_n*\bar{v}_n(t_2)\hat{\varphi}(t_2)dt_2dt_1+\delta_n(x,t)+T_{32}+T_3'.
\end{eqnarray}
The statement of Proposition \ref{T3} now follows from (\ref{T32bound}), (\ref{t32xt}), and (\ref{dddd}).
\end{proof}
To finish the proof of Proposition \ref{ynprop}, we observe that the equation (\ref{equationyn}) follows from (\ref{urav}), (\ref{uravn}), and 
Propositions \ref{T1}-\ref{T3}. Proposition \ref{ynprop} is proven.
\end{proof}


\subsection{The limit of $A_n$}
\label{an}
In this subsection, we study the limit of $A_n(t)$ as $n\to \infty.$ This, in turn, will allow us to study the limiting behavior of $Y_n(x,t).$ 
The main result of subsection is the following proposition.

\begin{prop}
\label{propan}
Let $A_n(t)$ be as defined in (\ref{An}).  Then the limit of $A_n(t)$ as $n\to \infty$ exists and 
\begin{eqnarray}
\label{atat}
A(t)&:=&\lim_{n\to \infty} A_n(t)\\
&=& -2\sigma^2\int_0^t \frac{1}{8\pi^3\sigma^2}\int_{-2\sqrt{2}\sigma}^{2\sqrt{2}\sigma}\int_{-2\sqrt{2}\sigma}^{2\sqrt{2}\sigma}e^{it_1x}\*\varphi'(y)\*
\sqrt{8\sigma^2-x^2}\*\sqrt{8\sigma^2-y^2}F_{\sigma}(x,y)\*1_{\{x\not=y\}} \*dxdydt_1, \nonumber
\end{eqnarray}
where for $x\not=y$
\begin{eqnarray}
F_{\sigma}(x,y)&=&\frac{\pi}{2\*\sigma^2}\sum_{k=0}^{\infty}U_k(x)U_k(y)\gamma_k\\
&=&\int_{-\infty}^{\infty}\frac{\frac{\sin^3 s}{s^3}-\frac{\sin s}{s}}{2\sigma^2
\left(1-\frac{\sin^2 s}{s^2}\right)^2-\left(\frac{\sin s}{s}+\frac{\sin^3 s}{s^3}\right)xy+\frac{\sin^2 s}{s^2}(x^2+y^2)}ds.
\end{eqnarray}
\end{prop}

\begin{proof}
We recall that $A_n(t)$ is defined in (\ref{An}) as $A_n(t)=-\frac{2\sigma^2}{n}\*\int_0^t\sum_{(j,k)\in I_n}\mathbb{E}\{U_{jk}(t_1)\varphi'(M)_{jk}\}dt_1.$ In the full Wigner matrix case, one has
$A_n=-\frac{2\sigma^2}{n}\*\int_0^t Tr [e^{it_1\*M} \* \varphi(M)] \* dt_1,$
and the limiting behavior of $A_n$ immediately follows from the Wigner semi-circle law.
In the band matrix case, there are additional difficulties due to the fact that the summation in the formula for $A_n$ is restricted  to
the band entries, i.e. to $(j,k)\in I_n.$

We start with the definition of a bilinear form on $C_b(\mathbb{R}),$  the space of bounded continuous functions on $\mathbb{R}.$
\begin{defn}
Let $f,g\in C_b(\mathbb{R})$. Define
\begin{equation}
<f,g>_n:=n^{-1}\mathbb{E}\{\sum_{(j,k)\in I_n}f(M)_{jk}\overline{g(M)_{jk}}\}.
\label{bilinear}
\end{equation}
\end{defn}
It follows from the above definition that
\begin{equation}A_n(t)=-2\sigma^2\int_0^t<e^{it_1x},\varphi'(x)>_ndt_1.
\end{equation}
The bilinear form (\ref{bilinear}) is an inner product (perhaps, degenerate). \\
(1)$ <f,f>_n\geq 0,$\\
(2)$<f,g>_n=\overline{<g,f>_n},$\\
(3)$<f,g_1+g_2>_n=<f,g_1>_n+<f,g_2>_n,\ \ <kf,g>_n=k<f,g>_n,\ k\in\mathbb{R},$\\
(4)(Cauchy-Schwarz Inequality) $|<f,g>_n|\leq <f,f>_n^{1/2}\*<g,g>_n^{1/2}.$\\

The proof of Proposition \ref{propan} relies on two auxiliary lemmas.
\begin{lem}
\label{lemma34}
For all $f,g\in C_b(\mathbb{R})$ the limit
\begin{equation}
\label{fgexists}
<f,g>:=\lim_{n\to\infty}<f,g>_n
\end{equation}
exists.
\end{lem}
\begin{proof}
We start with monomials. While monomials do not belong to $C_b(\mathbb{R})$, the expression (\ref{bilinear}) still makes sense since all moments of the 
matrix entries
of $M$ are finite.
For $l,m\in\mathbb{N},$ consider $f(x)=x^l, \ g(x)=x^m.$ Then
\begin{equation}
\label{dynamok}
<x^l,x^m>_n=\frac{1}{nb_n^{(l+m)/2}}\sum_{(i_0,i_1),...,(i_{l+m-1},i_0),(i_l,i_0)\in I_n}\mathbb{E}\{W_{i_0i_1}...W_{i_{l+m-1}i_0}\}.
\end{equation}
Let us fix $i_0\in\{1,...,n\}$. For $k=1,...,l+m,$ define
\begin{equation}
x_k=\left\{\begin{array}{lll}i_k-i_{k-1}&\mbox{if } |i_k-i_{k-1}|\leq b_n,\\i_k-i_{k-1}-n&\mbox{if } 
i_k-i_{k-1}>b_n,\\n+(i_k-i_{k-1})&\mbox{if } i_k-i_{k-1}<-b_n, 
\end{array}\right.
\end{equation}
where $i_{l+m}=i_0$. Since $l,m$ are fixed and $n/b_n\to\infty$, for sufficiently large $n$ the restriction 
$(i_0,i_1)$, $...$, $(i_{l+m-1},i_0)$, $(i_l,i_0)\in I_n$ is equivalent to $|x_1|,...,|x_{l+m}|\leq b_n, \ x_1+...+x_{l+m}=0,$ 
and $|x_1+...+x_l|\leq b_n$. 
Therefore, for sufficiently large $n,$
\begin{eqnarray}
<x^l,x^m>_n=\frac{1}{nb_n^{(l+m)/2}}\sum_{i_0=1}^n\sum_{\begin{array}{cccc}|x_1|,...,|x_{l+m}|\leq b_n\\x_1+x_2+...+x_{l+m}=0\\|x_1+x_2+...+x_l|
\leq b_n\end{array}}\mathbb{E}\{W_{i_0,i_1}...W_{i_{l+m-1},i_0}\}.
\end{eqnarray}
Each $(i_0,i_1,...,i_{l+m-1},i_0)$ is a closed path such that the distance between the endpoints of each edge is bounded by $b_n,$ 
and, in addition, the distance between $i_0$ and $i_l$ is also bounded by $b_n.$  
If $l+m$ is odd, one can show that $<x^l,x^m>_n\to0$ using power counting and independence of matrix entries. The proof is very similar 
to the combinatorial argument in the proof of the Wigner semicircle law and is left to the reader.

Now consider the case when $l+m$ is even.  Without loss of generality we can assume that $l\leq m$. 
As in the proof of the semicircle law, only the paths where every edge appears exactly twice 
contribute to the limit. For each such path, 
$$\mathbb{E}\{W_{i_0,i_1}...W_{i_{l+m-1},i_0}\}=\sigma^{l+m}.$$ 
Moreover, each such $(i_0,i_1,...,i_{l+m-1},i_0)$ corresponds to a Dyck path of length $l+m$ (see e.g. \cite{AGZ}).
Recall that a Dyck path $(s(0),...,s(l+m))$ of length $m+l$ satisfies
$$s(0)=s(l+m)=0,\ s(1),...,s(l+m-1)\geq 0, \ \text{and} \ |s(t+1)-s(t)|=1,\ i=0,...,l+m-1.$$
Specifically, $s(t+1)-s(t)=1$ if the non-oriented edge $(i_t, i_{t+1})$ appears in $(i_0,i_1,...,i_{l+m-1},i_0)$ for the first time and 
$s(t+1)-s(t)=-1$ if the edge $(i_t, i_{t+1})$ appears in $(i_0,i_1,...,i_{l+m-1},i_0)$ for the second time.

If  one removes in (\ref{dynamok}) the condition that $(i_l, i_0) \in I_n$ then the l.h.s. in (\ref{dynamok})
becomes $\frac{1}{n}\* Tr M^{l+m}$ and
each Dyck path gives equal contribution in the limit $n \to \infty.$  However, we have to take into account the condition $(i_l, i_0) \in I_n.$  As a 
result, the combinatorial analysis becomes more involved. 
Suppose $s(l)=k, \ 0\leq k\leq l$. Then during the first $l$ steps of the path $(i_0,i_1,...,i_{l+m-1},i_0),$ $(l-k)/2$ edges appear twice and $k$ 
edges appear only once. For each of the edges appearing twice, the corresponding two numbers $x_i$ have the same absolute value but differ in 
sign.  The remaining $k$ numbers $x_i$ will be renumerated (in the order of their appearance) by $y_1, y_2, \ldots y_k.$ 
One obtains
\begin{eqnarray}
& &<x^l,x^m>_n= \frac{\sigma^{l+m}}{b_n^{l+m}}\* \sum_{k=0}^l  \* \#\{ {\text Dyck} \ {\text paths} \ {\text of}\  {\text length}
\ l+m \ {\text with} \ s(l)=k \} \nonumber \\                                                                       
& &\times \#\mbox{of integeres}\{|y_1|\leq b_n, \ldots,|y_{k}|\leq b_n, \ldots, |y_{l+m}|\leq b_n, |y_1+...+y_k|\leq b_n\} \nonumber \\
& & +O(b_n^{-1}).
\end{eqnarray}

Therefore, $<x^l,x^m>=\lim_{n\to \infty} <x^l,x^m>_n$ exists, and
\begin{eqnarray}
& & <x^l,x^m>=   (\sqrt{2}\sigma)^{l+m}\sum_{k=0}^l
\* \#\{ {\text Dyck} \ {\text paths} \ {\text of}\  {\text length}
\ l+m \ {\text with} \ s(l)=k \} \nonumber \\ 
& & \times Vol\{|t_1|\leq 1/2, |t_2|\leq 1/2,\ldots,|t_{(l+m)/2}|\leq 1/2,|t_1+t_2+...+t_k|\leq 1/2\}\}.
\end{eqnarray}
The number of Dyck paths with $s(l)=k$ is 
\begin{equation}
\left[\binom{l}{\frac{l+k}{2}}-\binom{l}{\frac{l+k+2}{2}}\right]\left[\binom{m}{\frac{m+k}{2}}-\binom{m}{\frac{m+k+2}{2}}\right]=
\frac{(k+1)^2}{(l+1)(m+1)}\binom{l+1}{\frac{l+k+2}{2}}\binom{m+1}{\frac{m+k+2}{2}}.
\end{equation}
Let $T_1,...,T_{(l+m)/2}$ be i.i.d random variables uniformly distributed on $[-1/2,1/2]$. Then
\begin{equation}
Vol\{|t_1|\leq 1/2,...,|t_k|\leq 1/2,|t_1+t_2+...+t_k|\leq 1/2\}=\mathbb{P}(|T_1+T_2+...+T_k|\leq1/2).
\end{equation}
Let $S_k=T_1+...+T_k$. Then the characteristic function of $S_k$ is
\begin{equation}
\mathbb{E}\{e^{ixS_k}\}=(\mathbb{E}\{e^{ixY_1}\})^k=\left(\frac{\sin x/2}{x/2}\right)^k.
\end{equation}
Hence, the density function of $S_k$ is  given by
\begin{equation}
f_k(s)=\frac{1}{2\pi}\int_{-\infty}^{\infty} e^{-ixs}\left(\frac{\sin x/2}{x/2}\right)^kdx.
\end{equation}
Define $\gamma_k:=\mathbb{P}(|S_k|\leq 1/2)$. Then
\begin{equation}
\gamma_k=\int_{-1/2}^{1/2}f_k(s)ds=\frac{1}{2\pi}\int_{-\infty}^{\infty} \left(\frac{\sin x/2}{x/2}\right)^{k+1}dx=
\frac{1}{\pi}\int_{-\infty}^{\infty} \left(\frac{\sin x}{x}\right)^{k+1}\ dx=
f_{k+1}(0).
\end{equation}
The exact formula for $\gamma_k$ is well known (see e.g. \cite{methods}):
\begin{equation}\gamma_k=\left\{\begin{array}{cc}[(2t)!]^{-1}\sum_{s=0}^t(-1)^s\binom{2t+1}{s}(t-s+1/2)^{2t},&\mbox{ if } k=2t,\\
{[(2t+1)!]^{-1}}\sum_{s=0}^t(-1)^s\binom{2t+2}{s}(t-s+1)^{2t+1},&\mbox{ if }k=2t+1.\end{array}\right.\label{gamma}
\end{equation}
Therefore, we conclude that 
\begin{equation}
<x^l,x^m>=(\sqrt{2}\sigma)^{m+l}C_{l,m},
\label{bilinearform}
\end{equation}
where $C_{l,m}$ is defined in the following way.
For $l+m$ is odd, $C_{l,m}=0.$ For $l+m$ is even, $l\leq m$, 
\begin{eqnarray}
C_{l,m}=\left\{\begin{array}{ll}
[(l+1)(m+1)]^{-1}\sum_{k=0}^{l/2}(2k+1)^2\binom{l+1}{\frac{l-2k}{2}}\binom{m+1}{\frac{m-2k}{2}}\*\gamma_{2k},&\mbox{if } l \mbox{ is even,}\\
{[(l+1)(m+1)]^{-1}}\sum_{k=0}^{(l-1)/2}(2k+2)^2\binom{l+1}{\frac{l-2k-1}{2}}\binom{m+1}{\frac{m-2k-1}{2}}\gamma_{2k+1},&\mbox{if } l \mbox{ is odd.}
\end{array}\right.
\end{eqnarray}
For $l+m$ is even, $l> m$, 
$C_{l,m}=C_{m,l}.$ 
It follows from the definition that
$0\leq C_{l,m}\leq C_{\frac{l+m}{2}}$, 
where $C_s=\frac{(2s)!}{s!\*(s+1)!}$ is the Catalan number.

If $f,g$ are polynomials, $f(x)=\sum_{i=0}^pa_ix^i, g(x)=\sum_{j=0}^qb_jx^j$, then by linearity
\begin{equation}
<f,g>=\sum_{i=0}^p\sum_{j=0}^qa_ib_j(\sqrt{2}\sigma)^{i+j}C_{i,j}.
\end{equation}
Thus, the result of Lemma \ref{lemma34} holds when $f$ and $g$ are arbitrary polynomials.

For general bounded continuous functions $f,g$, we will show that $\{<f,g>_n\}$ is a Cauchy sequence.
To this end, we choose a sufficiently large $B$ independent of $n$ (it will be enough to take $B=4\sigma^3+1$ ).
Fix $\delta>0$. By the Stone-Weierstrass theorem, there exist polynomials $f_{\delta},g_{\delta}$ such that 
\begin{equation}
\sup_{x:|x|\leq B+1}|f(x)-f_{\delta}(x)|\leq\delta,\sup_{x:|x|\leq B+1}|g(x)-g_{\delta}(x)|\leq\delta.
\end{equation} 
Let $h$ be an infinitely differentiable function such that $ |h|\leq 1, \ \ h(x)=1$ for $|x|\leq B, \ \ h(x)=0$ for $|x|\geq B+1.$
We write
\begin{equation}
<f,g>_n=<f-f_{\delta}\*h,g-g_{\delta}\*h>_n+<f-f_{\delta}\*h,g_{\delta}\*h>_n+<f_{\delta},g-g_{\delta}\*h>_n+<f_{\delta}\*h,g_{\delta}\*h>_n.
\label{fgn}
\end{equation}
Below we show that the first three terms on the r.h.s. of (\ref{fgn}) are small provided $\delta$ is small.
It follows from (\ref{bilinear}) that
\begin{eqnarray}
\label{friday1}
& & <(f-f_{\delta})\*h, (f-f_{\delta})\* h>_n \leq \delta^2, \\
\label{friday2}
& & <(g-g_{\delta})\* h, (g-g_{\delta})\* h>_n \leq \delta^2.
\end{eqnarray}
Since $f, g$ are bounded on $\mathbb{R}$ and $f_{\delta}, g_{\delta} $ are polynomials, there exists sufficiently large 
$N\in\mathbb{Z}_+$, 
such that $(f-f_{\delta})^2 \* (1-h)\leq x^{2N} \* (1-h),$ 
and $(g-g_{\delta})^2 \* (1-h)\leq x^{2N} \* (1-h).$ 
Then
\begin{eqnarray}
& & <(f-f_{\delta})\* (1-h), (f-f_{\delta})\* (1-h)>_n \leq <x^{2N} \* (1-h), x^{2N} \* (1-h)>_n \nonumber \\
\label{friday3}
& & \leq \E \frac{1}{n\*B^{2N}} Tr M^{6N} \leq \delta^2
\end{eqnarray}
for sufficiently large $n,$ where the last inequality follows from the semicircle law provided $N$ is chosen so that 
$\frac{(\sqrt{2}\*\sigma)^{6N}}{B^{2N}}<\delta^2.$
In a similar fashion,
\begin{eqnarray}
\label{friday4}
& & <(g-g_{\delta})\* (1-h), (g-g_{\delta})\* (1-h)>_n \leq \delta^2, \\
\label{friday44}
& & <f\* (1-h), f\* (1-h)>_n \leq \delta^2, \ \ <g\* (1-h), g\* (1-h)>_n \leq \delta^2.
\end{eqnarray}
for sufficiently large $n.$  
The bounds (\ref{friday1}-\ref{friday44}) imply
\begin{eqnarray}
\label{friday5}
& & <f-f_{\delta}\*h, f-f_{\delta}\*h>_n \leq const\*\delta^2, \\
\label{friday6}
& & <g-g_{\delta}\*h, g-g_{\delta}\*h>_n \leq const\*\delta^2.
\end{eqnarray}
Now, applying (\ref{friday5}-\ref{friday6}) and the Cauchy-Schwarz inequality,
we obtain
\begin{eqnarray}
& & |<f,g>_n- <f_{\delta}\*h,g_{\delta}\*h>_n|\leq Const \*\delta, \nonumber \\
& & |<f_{\delta},g_{\delta}>_n- <f_{\delta}\*h,g_{\delta}\*h>_n|\leq Const \*\delta, \nonumber
\end{eqnarray}
and, as a result,
\begin{equation}
\label{fgn1}
|<f,g>_n- <f_{\delta},g_{\delta}>_n|\leq 2\* Const \*\delta.
\end{equation}
Therefore $<f,g>_n$ is a Cauchy sequence and $<f,g>$ exists. 
\end{proof}
In the next lemma, we diagonalize the bilinear form $<f,g>.$
\begin{lem}
\label{lemma35}
Let $\{U_n(x)\}$ be the (rescaled) Chebyshev polynomials of the second kind on $[-2\sqrt{2}\sigma,2\sqrt{2}\sigma]$,
\begin{equation}
U^{\sigma}_n(x)=\sum_{k=0}^{\lfloor n/2\rfloor}(-1)^k\binom{n-k}{k}\left(\frac{x}{\sqrt{2}\sigma}\right)^{n-2k}.
\end{equation}
Then $\{U^{\sigma}_n(x)\}_{n\geq 0}$ are orthogonal with respect to the bilinear form (\ref{bilinearform}), i.e.  
\begin{equation}
<U_n,U_m>=\delta_{nm}\gamma_n,
\label{or}
\end{equation}
where $\gamma_n$ is given by (\ref{gamma}). 
\end{lem}

\begin{rem}
{\it Note that $<f_1, g_1>=<f_2, g_2>$ if $f_1=f_2$ and $g_1=g_2$ in a neighborhood of $[-2\sqrt{2}\sigma,2\sqrt{2}\sigma].$  Thus, one can reformulate
Lemma \ref{lemma35} in such a way that $\{h\*U^{\sigma}_n\}_{n\geq 0}$ are orthogonal with respect to the bilinear form (\ref{bilinearform}).}
\end{rem}
\begin{rem}
\end{rem}
{\it Recall that the rescaled Chebyshev polynomials are orthonormal with respect to the Wigner semicircle law, i.e.}
\begin{equation}
\label{uhuh}
\int_{-2\sqrt{2}\sigma}^{-2\sqrt{2}\sigma} U^{\sigma}_n(x) \*  U^{\sigma}_m(x) \*\frac{1}{4 \pi \sigma^2} \sqrt{ 8 \sigma^2 - x^2} =\delta_{nm}.
\end{equation}
{\it Also,}
\begin{equation}
\label{sovwings}
U^{\sigma}_n(x)=\frac{\sin[(n+1)\*\theta]}{\sin \theta}, \ \ x=2\sqrt{2}\sigma \*\cos \theta.
\end{equation}
{\it When it does not lead to ambiguity, we will omit the super-index in the notation for the rescaled Chebyshev polynomials (alternatively, the reader 
can assume  that $2\sqrt{2}\sigma=1$). }
\begin{proof}
Since $<x^l,x^m>=0$ if $l+m$ is odd, it follows by linearity that
\begin{equation}
<U_n,U_m>=0, \mbox{ if } n+m \mbox{ is odd}. \label{eqn:odd}
\end{equation}
We are left to compute $<U_{2n},U_{2m}>$ and $<U_{2n+1},U_{2m+1}>$. We first compute $<x^{2l},U_{2n}>$ and \\
$<x^{2l+1},U_{2n+1}>$ for $l=0,1,...,n.$
One has
\begin{eqnarray}
<x^{2l},U_{2n}>&=&(\sqrt{2}\sigma)^{2l}\sum_{k=0}^n(-1)^k\binom{2n-k}{k}C_{2l,2n-2k}\nonumber\\
&=&\frac{(\sqrt{2}\sigma)^{2l}}{2l+1}\left[\sum_{k=0}^{n-l}\frac{(-1)^k}{2n-2k+1}\binom{2n-k}{k}\sum_{t=0}^l(2t+1)^2\binom{2l+1}{l-t}
\binom{2n-2k+1}{n-k-t}\gamma_{2t}\nonumber\right.\\
&&+\left.\sum_{k=n-l+1}^{n}\frac{(-1)^k}{2n-2k+1}\binom{2n-k}{k}\sum_{t=0}^{n-k}(2t+1)^2\binom{2l+1}{l-t}\binom{2n-2k+1}{n-k-t}
\gamma_{2t}\right]\nonumber\\
&=&\frac{(\sqrt{2}\sigma)^{2l}}{2l+1}
\sum_{t=0}^l(2t+1)^2\binom{2l+1}{l-t}
\left[\sum_{k=0}^{n-t}\frac{(-1)^k(2n-k)!}{k!(n-k-t)!(n-k+t+1)!}\right]\gamma_{2t},
\end{eqnarray}
and
\begin{eqnarray}
<x^{2l+1},U_{2n+1}>&=&(\sqrt{2}\sigma)^{2l+1}\sum_{k=0}^n(-1)^k\binom{2n+1-k}{k}C_{2l+1,2n+1-2k}\nonumber\\
&=&\frac{(\sqrt{2}\sigma)^{2l+1}}{2l+2}\left[\sum_{k=0}^{n-l}\frac{(-1)^k}{2n-2k+2}\binom{2n+1-k}{k}\sum_{t=0}^l(2t+2)^2\binom{2l+2}{l-t}
\binom{2n-2k+2}{n-k-t}\gamma_{2t+1}\nonumber\right.\\
&&+\left.\sum_{k=n-l+1}^{n}\frac{(-1)^k}{2n-2k+2}\binom{2n+1-k}{k}\sum_{t=0}^{n-k}(2t+2)^2\binom{2l+2}{l-t}\binom{2n-2k+2}{n-k-t}
\gamma_{2t+1}\right]\nonumber\\
&=&\frac{(\sqrt{2}\sigma)^{2l+1}}{2l+2}\sum_{t=0}^l(2t+2)^2\binom{2l+2}{l-t}
\left[\sum_{k=0}^{n-t}\frac{(-1)^k(2n+1-k)!}{k!(n-k-t)!(n-k+t+2)!}\right]\gamma_{2t+1}.
\end{eqnarray}
Denote
\begin{equation}
G_1(n,t)=\sum_{k=0}^{n-t}\frac{(-1)^k(2n-k)!}{k!(n-k-t)!(n-k+t+1)!},
\label{eqn:G1nt}
\end{equation}
\begin{equation}
G_2(n,t)=\sum_{k=0}^{n-t}\frac{(-1)^k(2n+1-k)!}{k!(n-k-t)!(n-k+t+2)!}.
\label{eqn:G2nt}
\end{equation}
Then
\begin{equation}
<x^{2l},U_{2n}>=\frac{(\sqrt{2}\sigma)^{2l}}{2l+1}\sum_{t=0}^l(2t+1)^2\binom{2l+1}{l-t}G_1(n,t)\gamma_{2t},
\label{evenor}
\end{equation}
\begin{equation}
<x^{2l+1},U_{2n+1}>=\frac{(\sqrt{2}\sigma)^{2l+1}}{2l+2}\sum_{t=0}^l(2t+2)^2\binom{2l+2}{l-t}G_2(n,t)\gamma_{2t+1}.
\label{oddor}
\end{equation}
It follows from (\ref{eqn:G1nt}-\ref{eqn:G2nt}) that
\begin{equation}G_1(n,t)=\frac{(2n)!}{(n-t)!(n+t+1)!}{}_2F_1\left(\begin{array}{c}-(n-t),-(n+t+1)\\-2n\end{array};1\right),\end{equation}
\begin{equation}G_2(n,t)=\frac{(2n+1)!}{(n-t)!(n+t+2)!}{}_2F_1\left(\begin{array}{c}-(n-t),-(n+t+2)\\-2n-1\end{array};1\right),\end{equation}
where $_2F_1$ is a hypergeometric function. By the Chu-Vandermonde identity (see e.g. \cite{specialfuntions}), we have
\begin{equation}_2F_1\left(\begin{array}{c}-(n-t),-(n+t+1)\\-2n\end{array};1\right)=\frac{(-n+t+1)_{n-t}}{(-2n)_{n-t}},\end{equation}
\begin{equation}_2F_1\left(\begin{array}{c}-(n-t),-(n+t+2)\\-2n-1\end{array};1\right)=\frac{(-n+t+1)_{n-t}}{(-2n-1)_{n-t}},\end{equation}
where $(a)_n=a(a+1)\cdots(a+n-1)$. Since 
\begin{equation}
(-n+t+1)_{n-t}=\left\{\begin{matrix}0&\mbox{if } t=0,1,...,n-1\\1&\mbox{if }t=n,\end{matrix}\right.
\end{equation}
we obtain
\begin{equation}G_1(n,t)=0,\ G_2(n,t)=0, \mbox{for } t=0,1,...,n-1,\end{equation}
and
\begin{equation}
G_1(n,n)=\frac{1}{2n+1},\ G_2(n,n)=\frac{1}{2n+2}.
\end{equation}
Therefore, for $l=0,1,\ldots, n-1,$ each term at the r.h.s. of (\ref{evenor}-\ref{oddor}) is zero, and
\begin{eqnarray}
&&<x^{2n},U_{2n}>=\frac{(\sqrt{2}\sigma)^{2n}}{2n+1}(2n+1)^2\*
\binom{2n+1}{0}\*G_1(n,n)\*\gamma_{2n}
=(\sqrt{2}\sigma)^{2n}\*\gamma_{2n}
\label{evenor2},\\
&&<x^{2n+1},U_{2n+1}>=\frac{(\sqrt{2}\sigma)^{2n+1}}{2n+2}(2n+2)^2\*
\binom{2n+2}{0}\*G_2(n,n)\*\gamma_{2n+1}
=(\sqrt{2}\sigma)^{2n+1}\*\gamma_{2n+1}.
\label{oddor2}
\end{eqnarray}
Hence, for $m< n$,
\begin{equation}
<U_{2m},U_{2n}>=0, \ <U_{2m+1},U_{2n+1}>=0,
\label{eqn:orthogonal1}
\end{equation}
and 
\begin{eqnarray}
&&<U_{2n},U_{2n}>=<\left(\frac{x}{\sqrt{2}\sigma}\right)^{2n},U_{2n}>=\gamma_{2n},
\label{or2}\\
&&<U_{2n+1},U_{2n+1}>=<\left(\frac{x}{\sqrt{2}\sigma}\right)^{2n+1},U_{2n+1}>=\gamma_{2n+1}.
\label{or3}
\end{eqnarray}
Combining (\ref{eqn:orthogonal1}), (\ref{eqn:odd}), (\ref{or2}) and (\ref{or3}), we complete the proof of Lemma \ref{lemma35}.
\end{proof}

Now, we are ready to finish the proof of Proposition \ref{propan}.
Let $f,g\in C_b(\mathbb{R})$, and
\begin{equation}
f_k=\frac{1}{4\pi\sigma^2}\int_{-2\sqrt{2}\sigma}^{2\sqrt{2}\sigma}f(x)U_k(x)\sqrt{8\sigma^2-x^2}dx,\ g_k=\frac{1}{4\pi\sigma^2}
\int_{-2\sqrt{2}\sigma}^{2\sqrt{2}\sigma}g(x)U_k(x)\sqrt{8\sigma^2-x^2}dx.
\end{equation}
Then
\begin{eqnarray}
\label{uuu}
<f,g>&=&\sum_{k=0}^{\infty}f_k\*g_k\gamma_k \\
\label{vvv}
&=&\frac{1}{8\*\pi^3\*\sigma^2}\int_{-2\sqrt{2}\sigma}^{2\sqrt{2}\sigma}\int_{-2\sqrt{2}\sigma}^{2\sqrt{2}\sigma}f(x)g(y)\sqrt{8\sigma^2-x^2}
\sqrt{8\sigma^2-y^2}\*F_{\sigma}(x,y)\* 1_{\{x\not=y\}}\*dxdydt,
\end{eqnarray}
where, for $x\not=y,$
\begin{eqnarray}
F_{\sigma}(x,y)&=&\frac{\pi}{2\*\sigma^2}\*\sum_{k=0}^{\infty}U_k(x)U_k(y)\gamma_k \nonumber\\
\label{fxyxy}
&=&  \int_{-\infty}^{\infty}\frac{\frac{\sin s}{s}-\frac{\sin^3 s}{s^3}}{2\sigma^2
\left(1-\frac{\sin^2 s}{s^2}\right)^2-\left(\frac{\sin s}{s}+\frac{\sin^3 s}{s^3}\right)xy+\frac{\sin^2 s}{s^2}(x^2+y^2)}ds.
\end{eqnarray}
Formula (\ref{uuu}) follows for polynomials from (\ref{or}) and (\ref{uhuh}), and then by continuity, by repeating the arguments at the end 
of the proof of Lemma \ref{lemma34}, for general continuous bounded functions.  Formula (\ref{fxyxy}) is a straightforward consequence of the 
Fourier analysis.
It follows from (\ref{sovwings}) that the r.h.s. of (\ref{uuu}) can be rewritten as 
\begin{equation}
\label{sto}
<f,g>= -2\*\sum_{l\not=0} \hat{\alpha}_l\*\hat{\beta}_l\*\gamma_{|l|-1},
\end{equation}
where
\begin{eqnarray}
& & \alpha(\theta)=f(2\sqrt{2}\*\sigma\*\cos\theta), \ \ 
\beta(\theta)=g(2\sqrt{2}\*\sigma\*\cos\theta), \\
& & \hat{\alpha}_l=\frac{1}{2\pi}\*\int_0^{2\pi} \alpha(\theta)\* e^{-i\*l\*\theta} \* d\theta,\ \ 
\hat{\beta}_l=\frac{1}{2\pi}\*\int_0^{2\pi} \beta(\theta)\* e^{-i\*l\*\theta}\* d\theta.
\end{eqnarray}
In particular, the trigonometric series $\sum_{l\not=0} \gamma_{|l|-1} \*e^{i\*l\*\theta}$ represents
an $L^1$ function $h$ which has $O(|\theta|^{-1/2})$  singularity near the origin.
The convergence is pointwise for all $\theta\not=0,$
\begin{eqnarray}
& & h(\theta)=\sum_{l\not=0} \gamma_{|l|-1} \*e^{i\*l\*\theta}, \ \theta\not=0, \nonumber\\
& & \hat{h}_l=\gamma_{|l|-1}, \ {\text if} \ l\not=0, \ \hat{h}_0=0. \nonumber
\end{eqnarray}
The convolution of $\beta$ and $h$ is then a continuous function on the unit circle, and 
one can rewrite (\ref{sto}) in the integral form by applying the Parseval's theorem.

Finally, it follows from (\ref{An}) and (\ref{vvv}) that
the limit of $A_n(x)$ exists and equals
\begin{equation}
A(t)=-2\sigma^2\int_0^t<e^{it_1x},\varphi'>dt_1
\end{equation}
with
\begin{eqnarray}
\label{bilinearlim}
<e^{it_1x},\varphi'>=\frac{1}{8\*\pi^3\sigma^2}\*\int_{-2\sqrt{2}\sigma}^{2\sqrt{2}\sigma}\*\int_{-2\sqrt{2}\sigma}^{2\sqrt{2}\sigma}\*e^{it_1x}
\*\varphi'(y) \*\sqrt{8\sigma^2-x^2}\*\sqrt{8\sigma^2-y^2}\*F_{\sigma}(x,y)\*1_{\{x\not=y\}}\*dx\*dy.
\end{eqnarray}
Proposition \ref{propan} is proven.
\end{proof}


\subsection{Variance}
The rest of the proof of Theorem \ref{thm-Poin} follows the steps in \cite{P.CLT}.
Using pre-compactness of $\{Y_n, Z_n\}_{n\geq 1},$ we consider a converging subsequence. Our goal is to show that the limit is unique.
Let \begin{equation}
Y_{n_j}(x,t)\to Y(x,t),\ Z_{n_j}(x)\to Z(x).
\end{equation}
By Wigner semicircle law,
\begin{equation}
\bar{v}_n(t)\to\int_{-2\sqrt{2}\sigma}^{2\sqrt{2}\sigma}\frac{e^{ity}}{4\pi\sigma^2}\sqrt{8\sigma^2-y^2}dy:=v(t).
\end{equation}
So the limit of $Y_n(x,t)$ satisfies the following equation:
\begin{eqnarray}
\lefteqn{Y(x,t)+4\sigma^2\int_0^t\int_0^{t_1}v(t_1-t_2)Y(x,t_2)dt_2dt_1=}\nonumber\\
&&xZ(x)A(t)+2i\kappa_4xZ(x)\int_0^tv*v(t_1)dt_1\int_{-\infty}^{\infty}t_2v*v(t_2)\hat{\varphi}(t_2)dt_2,
\label{avangard}
\end{eqnarray}
where $v*v$ is defined in (\ref{svertka}), 
\begin{equation}
v*v(t)=-\frac{i}{8\pi\sigma^4}\int_{-2\sqrt 2\sigma}^{2\sqrt2\sigma}e^{it\mu}\mu\sqrt{8\sigma^2-\mu^2}d\mu.
\end{equation}
Let
\begin{equation}
B=\int_{-\infty}^{\infty}t_2v*v(t_2)\hat{\varphi}(t_2)\*dt_2=\frac{1}{4\pi\sigma^4}\*
\int_{-2\sqrt{2}\sigma}^{2\sqrt{2}\sigma}\varphi(\mu)\frac{4\sigma^2-\mu^2}{\sqrt{8\sigma^2-\mu^2}}\*d\mu.
\end{equation}
As in \cite{P.CLT} (see formulas (2.82)-(2.86) and Proposition 2.1 there), 
we can solve (\ref{avangard}) to obtain
\begin{equation}
Y(x,t)=-\frac{2\sigma^2xZ(x)}{\pi}\int_{-2\sqrt{2}\sigma}^{2\sqrt{2}\sigma}\*\int_0^t\*\frac{e^{i\lambda (t-t_1)}<e^{it_1x},\varphi'>}
{\sqrt{8\sigma^2-\lambda^2}}\*\*dt_1\*d\lambda+\frac{i\kappa_4xZ(x)B}{4\pi\sigma^4}\*\int_{-2\sqrt{2}\sigma}^{2\sqrt{2}\sigma}\*\frac{e^{it\lambda}
(4\sigma^2-\lambda^2)}
{\sqrt{8\sigma^2-\lambda^2}}\*d\lambda.
\end{equation}

It then follows from (\ref{eqn:Znprime}) that
\begin{eqnarray}
Z'(x)&=&-\frac{2i\sigma^2xZ(x)}{\pi}\int_{-\infty}^{\infty}\int_{-2\sqrt{2}\sigma}^{2\sqrt{2}\sigma}
\*\int_0^t\frac{\hat{\varphi}(t)e^{i\lambda (t-t_1)}}{\sqrt{8\sigma^2-\lambda^2}}<e^{it_1x},\varphi'> \*dt_1 
\*d\lambda \*dt\nonumber\\
& & -\frac{\kappa_4\*x\*Z(x)}{16\*\pi^2\*\sigma^8}
\left(\int_{-2\sqrt{2}\sigma}^{2\sqrt{2}\sigma}\frac{\varphi(\lambda)(4\sigma^2-\lambda^2)}{\sqrt{8\sigma^2-\lambda^2}}d\lambda\right)^2.
\end{eqnarray}
One can rewrite the last formula in the form (\ref{eqn:Zn1}) with
\begin{eqnarray}
Var_{band}[\varphi]&=&\frac{2i\sigma^2}{\pi}\int_{-\infty}^{\infty}
\*\int_{-2\sqrt{2}\sigma}^{2\sqrt{2}\sigma}\*\int_0^t\frac{\hat{\varphi}(t)e^{i\lambda (t-t_1)}}{\sqrt{8\sigma^2-\lambda^2}}
<e^{it_1x},\varphi'>\*dt_1\*d\lambda \*dt \nonumber\\
&& +\frac{\kappa_4}{16\pi^2\sigma^8}
\left(\int_{-2\sqrt{2}\sigma}^{2\sqrt{2}\sigma}\frac{\varphi(\lambda)(4\sigma^2-\lambda^2)}{\sqrt{8\sigma^2-\lambda^2}}d\lambda\right)^2 \nonumber\\
\end{eqnarray}
Plugging in (\ref{bilinearlim}), finally we have
\begin{eqnarray}
Var_{band}[\varphi]&=&\int_{-2\sqrt{2}\sigma}^{2\sqrt{2}\sigma}\int_{-2\sqrt{2}\sigma}^{2\sqrt{2}\sigma}
\*\int_{-2\sqrt{2}\sigma}^{2\sqrt{2}\sigma}
\*\frac{(\varphi(x)-\varphi(\lambda))\varphi'(y)\sqrt{8\sigma^2-x^2}\sqrt{8\sigma^2-y^2}}{4\pi^4(x-\lambda)\sqrt{8\sigma^2-\lambda^2}}
F(x,y)\*1_{\{x\not=y\}}\*dx\*dy\*d\lambda \nonumber \\
& &+\frac{\kappa_4}{16\pi^2\sigma^8}
\left(\int_{-2\sqrt{2}\sigma}^{2\sqrt{2}\sigma}\frac{\varphi(\lambda)(4\sigma^2-\lambda^2)}{\sqrt{8\sigma^2-\lambda^2}}d\lambda\right)^2.
\end{eqnarray}
This finishes the proof of Theorem \ref{thm-Poin} for test functions satisfying (\ref{ineq:fourier}). 

Now, let $\varphi$ be an arbitrary function 
with bounded continuous derivative. It follows from Lemma \ref{norma} that we can assume that $\varphi$ has compact support inside the interval
$[-10\*\sigma, 10\*\sigma].$  One then approximates $\varphi$ in the $C_1([-10\*\sigma, 10\*\sigma])$ norm 
by functions satisfying (\ref{ineq:fourier}) and 
uses the bound (\ref{pi1}) to control the variance of the error term.  Theorem \ref{thm-Poin} is proven.


\section{Proof of Theorem \ref{thm:0cumulant}}

This section is devoted to the proof of Theorem \ref{thm:0cumulant}.
Thus, our goal is to extend the result of Theorem \ref{thm-Poin} to the case of non-i.i.d. entries with uniformly bounded fifth moment.
For technical reasons, we require that the fourth cumulant is zero and $\sqrt{n}\*\ln n \ll b_n.$
First, we establish two auxiliary lemmas.
\subsection{}
The first lemma is a simple statement about the norm of a sub-matrix of a unitary matrix.
\begin{lem}\label{lem:block}
Let $U$ be an $n\times n$ unitary matrix and $V$ be any $k\times k$ block of $U$. Then
\begin{equation}\|V\|\leq 1\end{equation}
\end{lem}
\begin{proof}
Suppose the indices of $V$ in $U$ are $(s,s+1,...,s+k-1)\times(t,t+1,...,t+k-1)$. Then 
\begin{equation}V=P_1UP_2\end{equation}
where $P_1$ is the orthogonal projection onto the subspace spanned by $e_s,...,e_{s+k-1}$, and $P_2$ is the orthogonal projection onto the subspace
spanned by $e_t,...,e_{t+k-1}.$
Then 
\begin{equation}
\|V\|\leq \|P_1\|\|U\|\|P_2\|= 1.
\end{equation}
\end{proof}
The second lemma gives an upper bound on  the norm of a band matrix built from a unitary matrix.
\begin{lem}
Let $U$ be an $n\times n$ unitary matrix. 
Let $b$ be a positive integer smaller than $n/2$. Denote
$I:=\{(j,k)| \ j,k=1,...,n,\ |j-k|\leq b \mbox{ or } n-|j-k|\leq b\}$. Let
\begin{equation}U^{(band)}:=\{U_{jk}, (j,k)\in I;0 \mbox{ otherwise}\}_{j,k=1}^n. \end{equation}
Then there exist positive constants $C_1$ and $C_2,$ independent from $n$ and $b,$ such that 
\begin{equation}
\|U^{(band)}\|\leq C_1\*\ln b+ C_2.
\label{ineq:Ubnorm}
\end{equation}
\end{lem}
\begin{proof}
Define
\begin{eqnarray}
& &A=\{U_{jk},\ |j-k|\leq b \ ;0 \mbox{ otherwise}\}_{j,k=1}^n,\\
& &B=\{U_{jk},\ n-|j-k|\leq b \ ;0 \mbox{ otherwise}\}_{j,k=1}^n.
\end{eqnarray}
Then 
\begin{equation}
U^{(band)}=A+B,
\end{equation}
and
\begin{equation}\|U^{(band)}\|\leq\|A\|+\|B\|\label{ineq:ABnorm}.\end{equation}
Matrix $B$ can be written as
\begin{equation}B=\left[\begin{array}{ccc}
0&0&B_1\\
0&0&0\\
B_2&0&0
\end{array}\right]\end{equation}
where $B_1,B_2$ are $(b+1)\times(b+1)$ matrices and $B_1 \ (B_2)$ is a strictly upper (lower) triangular matrix obtained from the corresponding
$(b+1)\times (b+1)$ block of $U$ by making all entries below (above) the main diagonal zero.
It is known (see e.g. \cite{operatornorm}) that if $B_{upper}$ is an upper triangular matrix constructed in such a way from an $N\times N$ matrix $B$ then
$\|B_{upper}\|\leq O(\log N)\*\|B\|.$  Applying
Lemma \ref{lem:block}, we obtain
\begin{eqnarray}
& & \|B_1\|,\|B_2\|\leq Const\ln(b+1), \nonumber \\
& & \|B\|\leq\|B_1\|+\|B_2\|\leq 2Const\ln(b+1). 
\label{ineq:Bnorm}
\end{eqnarray}
Now we turn our attention to the norm of $A.$ Write $n=m\times (b+1)-r, \ 0\leq r\leq b$.
Define
\begin{equation}U'=\left[\begin{array}{cccccccc}U&0_{n\times r}\\
0_{r\times n}&0_{r\times r}
\end{array}\right],A'=\left[\begin{array}{cccccccc}A&0_{n\times r}\\
0_{r\times n}&0_{r\times r}.
\end{array}\right]\end{equation}
Then $\|U'\|=\|U\|=1,\|A'\|=\|A\|$ and $A'$ can be written as a block matrix
\begin{equation}\left[\begin{array}{cccccccc}A_{11}&A_{12}&0&...&...&...\\
A_{21}&A_{22}&A_{23}&...&...&...\\
0&A_{32}&A_{33}&...&...&...\\
...&...&...&...&...&0\\
...&...&...&...&A_{m-1,m-1}&A_{m-1,m}\\
...&...&...&0&A_{m,m-1}&A_{m,m},
\end{array}\right]\end{equation}
where $A_{jj}'s$ are $(b+1)\times (b+1)$ blocks of $U'$.  Moreover, $A_{jk}'s \ (j\neq k)$ are strictly lower or upper triangular submatrices of the 
corresponding blocks in $U'$. Again, applying the Mathias bound in \cite{operatornorm},
we have
\begin{equation}
\|A_{jj}\|\leq1,\ \|A_{jk}\|\leq Const\ln (b+1).
\end{equation}
Let 
\begin{equation}
D=Diag\{A_{11},...,A_{mm}\}.
\end{equation}
Then \begin{equation}
\|D\|\leq \max_{1\leq i\leq m}\|A_{ii}\|\leq1.
\end{equation}
Let
\begin{equation}A_{i}=\left[\begin{array}{cccccccc}
0&A_{i,i+1}\\
A_{i+1,i}&0
\end{array}\right],i=1,...,m-1.
\end{equation} 
Then $A_i's$ are $2(b+1)\times 2(b+1)$ matrices, and
\begin{equation}\|A_i\|\leq\|A_{i,i+1}\|+\|A_{i+1,i}\|=2Const\ln(b+1).
\end{equation}
If $m$ is even, let
\begin{eqnarray*}
& & E=Diag\{A_1,A_3,...,A_{m-1}\}, \\
& & F=Diag\{0_{1\times1},A_2,A_4,...,A_{m-2},0_{1\times1}\}. 
\end{eqnarray*}
If $m$ is odd, let
\begin{eqnarray*}
& & E=Diag\{A_1,A_3,...,A_{m-2},0_{1\times1}\},\\
& & F=Diag\{0_{1\times1},A_2,A_4,...,A_{m-1}\}.
\end{eqnarray*}
Then
\begin{equation*}
A'=D+E+F,
\end{equation*}
and
\begin{equation}\|E\|,\|F\|\leq\max_{1\leq i\leq m}\{\|A_{i}\|\}=2\*Const\*\ln(b+1).
\end{equation}
Therefore, we have
\begin{equation}\|A'\|\leq \|D\|+\|E\|+\|F\|=1+4\*Const\*\ln(b+1)\leq C\*\ln(b+1).
\end{equation}

Therefore, 
\begin{equation}
\|A\|\leq C\*\ln(b+1).
\label{ineq:Anorm}
\end{equation}
Finally, (\ref{ineq:ABnorm}),(\ref{ineq:Bnorm}), and (\ref{ineq:Anorm}) imply (\ref{ineq:Ubnorm}).
\end{proof}

\subsection{}
Now, we are ready to prove Theorem \ref{thm:0cumulant}.
\begin{proof}[\bf{Proof of Theorem \ref{thm:0cumulant}.}]
Let $\hat{M}=b_n^{-1/2}\hat{W}$ be a band random real symmetric matrix with independent Gaussian random variables, 
and $M$ be an arbitrary band real symmetric random matrix satisfying the conditions in Theorem \ref{thm:0cumulant}.
We denote, respectively, by $\hat{\mathscr{M}}_n^{\circ}[\varphi]$
and $\mathscr{M}_n^{\circ}[\varphi]$
the centered normalized linear eigenvalue statistics of $\hat{M}$ and $M$ defined as in (\ref{linejnaya}).
Since Gaussian distribution satisfies the Poincar\'e inequality, Theorem \ref{thm-Poin} establishes the Central Limit Theorem for 
$\hat{\mathscr{M}}_n^{\circ}[\varphi].$
Thus, it suffices to show that, for every $x\in\mathbb{R}$, 
\begin{equation}
R_n(x):=\mathbb{E}\{e^{ix\mathscr{M}_n^{\circ}[\varphi]}\}-\mathbb{E}\{e^{ix\hat{\mathscr{M}}_n^{\circ}[\varphi]}\}\to 0, \ n\to\infty.
\label{eqn:defRn}
\end{equation}
Let us denote 
\begin{equation}
e_n(s,x)=\exp\{(b_n/n)^{1/2}ixTr\varphi(M(s))^{\circ}\},
\end{equation}
where $M(s)$ is the interpolating matrix $M(s)=s^{1/2}M+(1-s)^{1/2}\hat{M},\ 0\leq s\leq 1.$ 
We have 
\begin{equation}
R_n(x)=\int_0^1\frac{\partial}{\partial s}\mathbb{E}\{e_n(s,x)\}ds.
\end{equation}
Taking into account that
\begin{eqnarray}
\frac{\partial}{\partial s}e_n(s,x)&=&\sum_{(j,k)\in I_n^+}\frac{\partial e_n(s,x)}{\partial M_{jk}(s)}\frac{\partial M_{jk}(s)}{\partial s}\nonumber\\
&=&(b_n/n)^{1/2}ixe_n(s,x)\sum_{(j,k)\in I_n^+}\frac{\partial Tr\varphi(M(s))^{\circ}}{\partial M_{jk}(s)}\frac{\partial M_{jk}(s)}{\partial s}\nonumber\\
&=&(b_n/n)^{1/2}ixe_n(s,x)\sum_{(j,k)\in I_n^+}2\beta_{jk}(\varphi'_{jk}(M(s)))^{\circ}\frac{1}{2}(s^{-1/2}M_{jk}-(1-s)^{-1/2}\hat{M}_{jk})\nonumber\\
&=&\frac{\sqrt{b_n}}{2\sqrt{n}}ixe_n(s,x)Tr(\varphi'(M(s)))^{\circ}(s^{-1/2}M-(1-s)^{-1/2}\hat{M}),
\end{eqnarray}
we can write
\begin{equation}
\label{zzzzz}
R_n(x)=\frac{ix}{2\sqrt{n}}\int_0^1\mathbb{E}\{e_n^{\circ}(s,x)Tr\varphi'(M(s))(s^{-1/2}W-(1-s)^{-1/2}\hat{W})\}ds.
\end{equation}
Since
\begin{equation}
\varphi'(M)=i\*\int_{-\infty}^{+\infty}\hat{\varphi}(t)\*t\*U(t)\*dt,
\end{equation}
we can rewrite (\ref{zzzzz}) as
\begin{eqnarray}
R_n(x)&=&-\frac{x}{2\sqrt{n}}\*\int_0^1\*\int\hat{\varphi}(t)\*t\*\mathbb{E}\*\{e_n^{\circ}(s,x)\*Tr\*U(s,t)\*(s^{-1/2}W-(1-s)^{-1/2}\hat{W})\}\*dt\*ds\\
&=&-\frac{x}{2}\*\int_0^1 \*\int \*\hat{\varphi}(t)\*t\*[K_n-L_n]\*dt\*ds,
\label{eqn:Rnint}
\end{eqnarray}
where
\begin{eqnarray}
& & K_n=\frac{1}{\sqrt{ns}}\sum_{(j,k)\in I_n}\mathbb{E}\{W_{jk}\Phi_n\}, \nonumber \\
& & L_n=\frac{1}{\sqrt{n(1-s)}}\sum_{(j,k)\in I_n}\mathbb{E}\{\hat{W}_{jk}\Phi_n\}, \ \ {\text and}\nonumber \\
& & \Phi_n=U_{jk}(s,t)e_n^{\circ}(s,x),\ \ U(s,t)=e^{itM(s)}. \nonumber
\end{eqnarray}
Applying the decoupling formula with $p=3$ to every term in $K_n$ and $L_n$, we obtain
\begin{equation}
K_n-L_n=I_2+I_3+\varepsilon_3,
\label{eqn:An-Bn}
\end{equation}  
where
\begin{equation}
I_l=\frac{s^{(l-1)/2}}{l!n^{1/2}b_n^{l/2}}
\sum_{(j,k)\in I_n}\kappa_{l+1,jk}\mathbb{E}\{D_{jk}^l(s)\Phi_n\},D_{jk}(s)\partial/\partial M_{jk}(s),\ l=2,3,
\end{equation}
and
\begin{equation}
|\varepsilon_3|\leq\frac{C_3\sigma_5}{\sqrt{n}b_n^2}\sum_{(j,k)\in I_n}\sup_{M\in\mathbb{R}}|D_{jk}^4(s)\Phi_n|_{M(s)=M}|.
\end{equation}
Let us consider $I_2$ first.
\begin{eqnarray}
I_2&=&\frac{\sqrt{s}\kappa_3}{n^{3/2}}\*2\*x^2\*\sum_{(j,k)\in I_n}\*\beta_{jk}^2\*
\mathbb{E}\*\{e_n(s,x)\*U_{jk}(s,t)\left[\int \theta\*\hat{\varphi}(\theta)\*U_{jk}(s,\theta)\*d\theta\right]^2\}\nonumber\\
&&-\frac{\sqrt{s}\kappa_3}{n\sqrt{b_n}}\*i\*x\*\sum_{(j,k)\in I_n}\beta_{jk}^2\mathbb{E}\{e_n(s,x)U_{jk}(s,t)\*\int \theta\*\hat{\varphi}(\theta)
[U_{jj}*U_{kk}+U_{jk}*U_{jk}](s,\theta)\*d\theta\}\nonumber\\
&&-\frac{\sqrt{s}\kappa_3}{n\sqrt{b_n}}\*x\*\sum_{(j,k)\in I_n}\beta_{jk}^2\mathbb{E}\{e_n(s,x)[U_{jj}*U_{kk}+U_{jk}*U_{jk}](s,t)\*\int \*\theta
\*\hat{\varphi}(\theta)U_{jk}(s,\theta)\*d\theta\}\nonumber\\
&&-\frac{\sqrt{s}\kappa_3}{\sqrt{n}b_n}\sum_{(j,k)\in I_n}\beta_{jk}^2\*\mathbb{E}\{e_n^{\circ}(s,x)(U_{jk}*U_{jk}*U_{jk}+3U_{jj}*U_{kk}*U_{jk})(s,t)\}+I_2',
\end{eqnarray}
where 
\[I_2'=\frac{\sqrt{s}}{2\sqrt{nb_n}}
\sum_{j=1}^n(\kappa_{3,jj}-\kappa_3)\mathbb{E}\{D_{jk}^2(s)\Phi_n\}.\]
Recall $\kappa_{3,jj}$ is the third cumulant of the $j$th diagonal entry and $\kappa_3$ is the third cumulant of the off-diagonal entries. Note that
\begin{equation}|D_{jk}^l(s)\Phi_n|\leq C_l(\sqrt{b_n/n}x,t),0\leq l\leq 4,
\end{equation}
So 
\[|I_2'|\leq \frac{\sqrt{n}}{b_n}C_2(\sqrt{b_n/n}x,t).\]
Consider two types of the sums above:
\begin{equation}I_{21}=\sum_{(j,k)\in I_n}U_{jj}(s,t_1)U_{jk}(s,t_2)U_{kk}(s,t_3),\end{equation}
\begin{equation}I_{22}=\sum_{(j,k)\in I_n}U_{jk}(s,t_1)U_{jk}(s,t_2)U_{jk}(s,t_3).\end{equation}
It follows from the Cauchy-Schwarz inequality that
\begin{equation}|I_{22}|\leq \left(\sum_{j,k=1}^n|U_{jk}(s,t_1)|^2\right)^{1/2}\left(\sum_{j,k=1}^n|U_{jk}(s,t_2)|^2\right)^{1/2}=n.\end{equation}
In addition,
\begin{equation}
I_{21}=n(U^{(B)}(s,t_2)V(t_1),V(t_3)),V(t)=n^{-1/2}(U_{11}(t),...,U_{nn}(t))^t.
\end{equation}
Since $\|U^{(B)}\|\leq C\ln b_n,\ \|V(t)\|\leq1$, we have $|I_{21}|\leq C\*n\*\ln b_n.$
Therefore,
\[|I_2|\leq C_1\frac{x^2}{\sqrt{n}}+C_2\frac{|x|\ln b_n}{\sqrt{n}}+C_3\frac{\sqrt{n}\ln b_n}{b_n}+\frac{\sqrt{n}}{b_n}C_2(\sqrt{b_n/n}x,t).\]
Since $\frac{\sqrt{n}\ln n}{b_n}\to 0$, then $I_2\to0$ on any bounded subset of $\{(x,t)|t\geq0\}$.

Recall that $\kappa_{4,jk}=0,\ j\neq k.$  Thus,
\begin{equation}
I_3=\frac{s}{3!n^{1/2}b_n^{3/2}}\sum_{j=1}^n\kappa_{4,j}\mathbb{E}\{D_{jj}^3(s)\Phi_n\},
\end{equation}
and
\begin{equation}
|I_3|\leq \frac{\sqrt{n}}{b_n^{3/2}}C_3(\sqrt{b_n/n}x,t).
\end{equation}
Taking into account that
\begin{equation}
|\varepsilon_3|\leq \frac{\sqrt{n}}{b_n}C_4(\sqrt{b_n/n}x,t), 
\end{equation}
we conclude that $I_2,I_3, \varepsilon \to0$ on any bounded subset of $\{(x,t):t\geq0\}$.
It then follows from (\ref{eqn:An-Bn}) and (\ref{eqn:Rnint}), that $R_n,$ defined in (\ref{eqn:defRn}), converges to 0 as $n\to\infty$.
Theorem \ref{thm:0cumulant} is proven.
\end{proof}

\section{Appendix}
\appendix
\section{Poincar\'e Inequality}
\begin{defn}
A probability measure P on $\mathbb{R}^M$ satisfies the Poincar\'e Inequality (PI) with constant $m>0$ if, for all continuously differentiable 
functions $f$,
\begin{equation}
Var_P(f):=E_P(|f(x)-E_P(f(x))|^2)\leq \frac{1}{m}E_P(|\nabla f|^2).\end{equation}\label{pi}
\end{defn}
We note that the Poincar\'e inequality tensorises and the probability measures satisfying the Poincar\'e inequality have sub-exponential tails
(see e.g. \cite{AGZ}).  In particular, 
if $P$ satisfies the PI on $\mathbb{R}^M$ with constant $m,$
then for any Lipschitz continuous function $G$, and $|t|\leq \sqrt{m}/\sqrt{2}|G|_{\mathscr{L}},$ we have
\begin{equation}E_{P}(e^{t(G-E_P(G))})\leq K,
\end{equation}
with $K=-\sum_{i\geq0}2^i\log(1-2^{-1}4^{-i})$. Consequently, for all $\delta>0$,
\begin{equation}P(|G-E_P(G)|\geq\delta)\leq 2Ke^{-\frac{\sqrt{m}}{\sqrt{2}|G|_{\mathscr{L}}}\delta}.\label{poincarelem}\end{equation}

\section{Decoupling formula} 
\begin{defn}
Let $\xi$ be a random variable such that $\mathbb{E}\{|\xi|^{p+2}\}<\infty$ for a certain nonnegative integer p. Then for any function $f:\mathbb{R}\to 
\mathbb{C}$ of the class $C^{p+1}$ with bounded derivatives $f^{(l)},l=1,...,p+1,$ we have
\begin{equation}
\mathbb{E}\{\xi f(\xi)\}=\sum_{l=0}^p\frac{\kappa_{l+1}}{l!}\mathbb{E}\{f^{(l)}(\xi)\}+\varepsilon_p.
\label{decf}
\end{equation}
where $\kappa_l$ denotes the $l$th cumulant of $\xi$ and the remainder term $\varepsilon_p$ admits the bound
\begin{equation}
|\varepsilon_p|\leq C_p\mathbb{E}\{|\xi|^{p+2}\}\sup_{t\in \mathbb{R}}f^{(p+1)}(t), \ C_p\leq\frac{1+(3+2p)^{p+2}}{(p+1)!}.
\label{epsilon}
\end{equation}
If $\xi$ is a Gaussian random variable with zero mean,
\begin{equation}
\mathbb{E}\{\xi f(\xi)\}=\mathbb{E}\{\xi^2\}\mathbb{E}\{f'(\xi)\}.
\label{dfgaus}
\end{equation}
\end{defn}
\section{Proof of Proposition \ref{propprop}}

The goal of this section is to derive a bound
\begin{equation}
\sup _{j\neq k}  |\mathbb{E}\{U_{jk}(t)\}|=O(\frac{1+t^6}{b_n}).
\end{equation}
To achieve this, we first bound the mathematical expectation of the off-diagonal entries of the resolvent matrix.  Then, we use the
Helffer-Sj\'{o}strand functional calculus to extend the bound to the off-diagonal entries of the unitary matrix $U(t).$ 

Consider $R(z)=(z-M)^{-1}$, $Im(z)\neq 0$. 
The main part of the proof of proposition is the following lemma.
\begin{lem}
\label{lemma}
Let $|Imz|\leq 2$. Then
\begin{equation}
|\mathbb{E}\{R_{ps}\}| \leq \frac{C}{|\Im z|^5\*b_n},
\end{equation}
where $C>0$ is a constant independent from
$p\not=s$ and $n.$
\end{lem}
\begin{proof}
We start with the resolvent identity 
\begin{equation}zR(z)=I+M\*R(z).\end{equation}
Therefore, the off-diagonal entries of $R(z)$ satisfy the following equation
\begin{equation}
z\mathbb{E}\{R_{ps}\}=\sum_{j: (j,p)\in I_n}\mathbb{E}\{M_{pj}R_{js}\}, \ p\not=s.
\label{zR_ps}
\end{equation}
Applying the decoupling formula, we obtain
\begin{eqnarray}
\lefteqn{\mathbb{E}\{M_{pj}R_{js}\}=}\nonumber\\
&&\left\{\begin{array}{llll}
\frac{\sigma^2}{b_n}\mathbb{E}\{R_{jp}R_{js}+R_{jj}R_{ps}\}+\frac{\mu_3}{b_n^{3/2}}\mathbb{E}\{2R_{jp}^2R_{js}+2R_{jj}R_{pp}R_{js}+
4R_{jp}R_{jj}R_{ps}\}+\varepsilon_{2,j}&j\neq p\\
\frac{2\sigma^2}{b_n}\mathbb{E}\{R_{pp}R_{ps}\}+\frac{2\mu_{3,p}}{b_n^{3/2}}\mathbb{E}\{R_{pp}^2R_{ps}\}+\varepsilon_{2,p}&j=p,
\end{array}\right.
\end{eqnarray}
where
\begin{equation}
|\varepsilon_{2,j}|\leq \frac{C_2\max\{\kappa_4,\kappa_4'\}}{b_n^2}\sup_{M_{pj}\in\mathbb{R}}|\frac{\partial^3R_{js}}{\partial M_{pj}^3}|
=O(\frac{1}{b_n^2|Imz|^4}).
\end{equation}
We note that
\begin{eqnarray}
& & |\sum_{j: (j,p)\in I_n}\mathbb{E}\{R_{jp}R_{js}\}|\leq\mathbb{E}\{\sqrt{\sum_{|j-p|\leq b_n}|R_{jp}|^2}
\sqrt{\sum_{|j-p|\leq b_n}\mathbb{E}\{|R_{js}|^2}\}\leq\frac{1}{|Imz|^2}, \\
& & |\sum_{j: (j,p)\in I_n} \mathbb{E}\{R_{jp}^2R_{js}\}|\leq\frac{1}{|Imz|^2}\sum_{j:|j-p|\leq b_n}\mathbb{E}\{|R_{js}|\}\leq
\frac{\sqrt{2b_n+1}}{|Imz|^2}\sqrt{\sum_{j:|j-p|\leq b_n}|R_{js}|^2}\leq\frac{\sqrt{2b_n+1}}{|Imz|^3},
\end{eqnarray}
and similarly,
\begin{eqnarray}
& |\sum_{j: (j,p)\in I_n}\mathbb{E}\{R_{jj}R_{pp}R_{js}\}|\leq\frac{\sqrt{2b_n}}{|Imz|^3}, 
& |\sum_{j: (j,p)\in I_n}
\mathbb{E}\{R_{jp}R_{jj}R_{ps}\}|\leq\frac{\sqrt{2b_n}}{|Imz|^3}.
\end{eqnarray}
Thus, for $p\neq s,$
\begin{eqnarray}
z\mathbb{E}\{R_{ps}\}&=&\sum_{j: (j,p)\in I_n}\frac{\sigma^2}{b_n}\mathbb{E}\{R_{jj}R_{ps}\}+O(\frac{1}{b_n|Imz|^2})
+O(\frac{1}{b_n|Imz|^3})+O(\frac{1}{b_n|Imz|^4})\nonumber\\
&=&\sum_{j: (j,p)\in I_n }\frac{\sigma^2}{b_n}\mathbb{E}\{R_{jj}R_{ps}\}+O(\frac{1}{b_n|Imz|^4}).
\end{eqnarray}
Since the diagonal entries $R_{jj}$'s have the same distribution, 
we can write
$g_n(z):=\frac{1}{n}\mathbb{E}\{TrR\}=\mathbb{E}\{R_{jj}\}$. From the Wigner semicircle law for band random matrices, 
\begin{equation}
g_n(z)\to\int_{-2\sqrt{2}\sigma}^{2\sqrt{2}\sigma}\frac{\sqrt{8\sigma^2-x^2}}{4\pi\sigma^2(z-x)}dx.
\end{equation}
We have
\begin{equation}\sum_{j: (j,p)\in I_n}\mathbb{E}\{R_{jj}R_{ps}\}=(2b_n+1)g_n(z)\mathbb{E}\{R_{ps}\}+\sum_{|j-p|\leq b_n}
\mathbb{E}\{R_{jj}^{\circ}R_{ps}^{\circ}\},
\end{equation}
and
\begin{equation}|\sum_{j: (j,p)\in I_n}\mathbb{E}\{R_{jj}^{\circ}R_{ps}^{\circ}\}|\leq(2b_n+1)Var^{1/2}\{R_{11}\}Var^{1/2}\{R_{ps}\}.
\end{equation}
The Poincar\'e inequality implies that
\begin{equation}Var\{R_{ps}\}\leq\frac{1}{mb_n}\sum_{j: (j,p)\in I_n}
\mathbb{E}\{\beta_{jk}^2|R_{pj}R_{ks}+R_{pk}R_{js}|^2\}\leq\frac{2}{mb_n|Imz|^4}.
\end{equation}
Hence,
\begin{equation}z\mathbb{E}\{R_{ps}\}=\frac{\sigma^2(2b_n+1)}{b_n}g_n(z)\mathbb{E}\{R_{ps}\}+O(\frac{1}{|\Im z|^4 \* b_n}),
\end{equation}
which implies
\begin{equation}[z-\frac{2b_n+1}{b_n}\sigma^2g_n(z)]\mathbb{E}\{R_{ps}\}=O(\frac{1}{|\Im z|^4\*b_n}).
\label{ER_ps=O}
\end{equation}
In a similar fashion,
\begin{equation}
\label{bavaria1}
[z-\frac{2b_n+1}{b_n}\sigma^2g_n(z)]g_n(z)=1+O(\frac{1}{|\Im z|^4 \*b_n}).
\end{equation}

If the term  $O(\frac{1}{|\Im z|^4 \*b_n})$ at the r.h.s. of (\ref{bavaria1}) is bounded in absolute value from above by $1/2,$ then
there exists a constant $C_1$ such that
\begin{equation}
\frac{C_1}{b_n|Imz|^4}\leq 1/2.
\end{equation}
Then
\begin{eqnarray}
& & |[z-\frac{2b_n+1}{b_n}\sigma^2g_n(z)]g_n(z)|\geq 1/2, \\
\label{bavaria2}
& & |z-\frac{2b_n+1}{b_n}\sigma^2g_n(z)|\geq \frac{1}{2|g_n(z)|}\geq\frac{|Imz|}{2},
\end{eqnarray}
and (\ref{bavaria2}) and (\ref{ER_ps=O}) imply
\begin{equation}|\mathbb{E}\{R_{ps}\}|=O(\frac{1}{|\Im z|^5\* b_n}).
\end{equation}
Now assume that
\begin{equation}
\frac{C_1}{b_n|Imz|^4}>1/2.
\end{equation}
Then
\begin{equation}
|\mathbb{E}\{R_{ps}\}|\leq\frac{1}{|Imz|}<\frac{2\*C_1}{b_n|Imz|^5}.
\label{resres}
\end{equation}
Lemma \ref{lemma} is proven. \end{proof}

Now, we extend the bound in the last lemma to the off-diagonal entries of $f(M),$ where $f$ is sufficiently smooth function with compact support.
To this end, we use the Helffer-Sj\"ostrand functional calculus (see e.g. \cite{Dav}, \cite{sashapizzo}).
We write
\begin{equation}
\label{bound12}
\mathbb{E}\{f(M)_{jk}\}=-\mathbb{E}\left\{\frac{1}{\pi}\int_{\mathbb{C}}\frac{\partial\tilde{f}}{\partial\bar{z}}R_{jk}dxdy\right\}
=-\frac{1}{\pi}\int_{\mathbb{R}\times [-1,1]}\frac{\partial\tilde{f}}{\partial\bar{z}}O(\frac{1}{|\Im z|^5\* b_n})dxdy,
\end{equation}
where
\begin{itemize}
\item[i)]
 $z=x+iy$ with $x,y \in \mathbb{R}$;
 \item[ii)] $\tilde{f}(z)$ is the extension of the function $f$ defined as
  \begin{equation}\label{a.a. -extension}
  \tilde{f}(z):=\Big(\,\sum_{n=0}^{l}\frac{f^{(n)}(x)(iy)^n}{n!}\,\Big)\sigma(y);
\end{equation}
here $\sigma \in C^{\infty}(\mathbb{R})$ is a nonnegative function equal to $1$ for $|y|\leq 1/2$ and equal to zero for $|y|\geq 1$.
 \end{itemize}
Since
\begin{equation}
\frac{\partial \tilde{f}}{\partial \bar{z}} := \frac{1}{2}\Big(\frac{\partial \tilde{f}}
{\partial x}+i\frac{\partial \tilde{f}}{\partial y}\Big),
\end{equation}
one has (with $l=5$)
\begin{equation}
\frac{\partial\tilde{f}}{\partial\bar{z}}=
\frac{1}{2}\left(\sum_{n=0}^5\frac{f^{(n)}(x)(iy)^n}{n!}\right)i\frac{d\sigma}{dy}+\frac{1}{2}f^{(6)}(x)(iy)^5\frac{\sigma(y)}{5!}.
\end{equation}
In particular,
\begin{equation}\label{estimate-derivative}
\Big|\frac{\partial \tilde{f}}{\partial \bar{z}}\Big| \leq const \* \|f\|_{C_c^6(\R)}\*|y|^5,
\end{equation}
for six times continuously differentiable function $f$ with compact support, where
\begin{equation}
\|f\|_{C_c^6(\R)}=\max_{0\leq k\leq 6} \max_{x \in \R} |f^{(k)}(x)|.
\label{bound17}
\end{equation}

Combining (\ref{resres}), (\ref{bound12}), and (\ref{bound17}), we arrive at
\begin{equation}
\mathbb{E}\{f(M)_{jk}\}=O(\frac{\|f\|_{C_c^6(\R)}}{b_n}).
\end{equation}
This bound is not sufficient for our purposes since $g(x)=e^{itx}$ is not compactly supported. 
Let $f(x)\in C^{\infty}(\mathbb{R})$ be a function satisfying $f(x)\equiv g(x)$ if $x\in[-10\sigma,10\sigma]$, $f(x)=0$ if $|x|>20\sigma$. 
If $Spec(M)\subset[-10\sigma,10\sigma]$, we clearly have $f(M)=g(M)$. Hence,
\begin{equation}
|\mathbb{E}\{g(M)_{jk}\}|\leq|\mathbb{E}\{f(M)_{jk}\}|+\sup_{x\in\mathbb{R}}|g(x)|\mathbb{P}(\|M\|\geq10\sigma).
\label{gexpectation}
\end{equation}
In the next lemma, we show that $\mathbb{P}(\|M\|\geq10\sigma)$ is negligibly small.

\begin{lem}
\label{norma}
There exists a positive constant $C$  such that
\begin{equation}
\label{otsenka10}
\mathbb{P}(\|M\|\geq10\sigma)\leq C \* e^{-C \*\sqrt{b_n}\*\sigma}.
\end{equation}
\end{lem}
Clearly, (\ref{gexpectation}) and (\ref{norma}) finish the proof of Proposition \ref{propprop}.
Thus, we are left with proving (\ref{otsenka10}).

\begin{proof}
We note that $\|M\|$ is a Lipschitz function of the matrix entries and the distribution of the entries of $M$ satisfies the Poincar\'e 
inequality.  Therefore, we have 
\begin{equation}
\mathbb{P}(|\|M\|-\mathbb{E}\{\|M\|\}|\geq\delta)\leq c_1e^{-c_2\sqrt{b_n}\delta},
\label{normpoin}
\end{equation}
with some positive constants $c_1$ and $c_2.$
Below we show that $\mathbb{E}\{\|M\|\} \leq 5\sigma$ for all sufficiently large $n.$

Let $\tilde{M}$ be an independent copy of $M$.  Using a symmetrization argument (see e.g. \cite{tao}),
we have
\begin{equation}
\mathbb{E}\{\|M -\tilde{M}\|\}\geq \mathbb{E}\{\|M\|\}
\end{equation}
Denote $B=M-\tilde{M}.$
Applying the method of moments (\cite{sashasinai1}, \cite{sashasinai2}), one can show that
\begin{equation}
\label{bound11}
\mathbb{E}\{Tr B^{2s}\}=\frac{(16\sigma^2)^sn}{\sqrt{\pi s^3}}(1+o(1)),
\end{equation}
as $n\to\infty$ provided $s\to\infty$ so that $s=o(b_n^{1/3})$.
The computations are standard and left to the reader.
Then
\begin{equation}
\mathbb{E}\{\|B\|^{2s}\}\leq\frac{(16\sigma^2)^sn}{\sqrt{\pi s^3}}(1+o(1)),
\end{equation}
which implies
\begin{equation}
\mathbb{E}\{\|B\|\}\leq 4\sigma\left[\frac{n}{\sqrt{\pi s^3}}(1+o(1))\right]^{1/2s}.
\end{equation}
Therefore, for sufficiently large $n,$
\begin{equation}
\mathbb{E}\{\|M\|\} \leq \mathbb{E}\{\|B\|\} \leq 5\*\sigma.
\end{equation}
The last inequality and (\ref{normpoin}) finish the proof of Lemma \ref{norma} and Proposition \ref{propprop}.
\end{proof}


\bibliographystyle{plain}


\end{document}